\documentclass[11pt,letterpaper,oneside]{amsart}

  \usepackage{amsmath,amssymb,amsthm}
  \usepackage{amscd}
  \usepackage{geometry}
  \usepackage{graphicx,epstopdf}
  \usepackage{color}
  \usepackage{enumerate}
  \usepackage{xspace}
  \usepackage{tipa}
\usepackage{latexsym}
  \usepackage{epsfig}
  \usepackage{pinlabel}
  \usepackage[all]{xy}

  \usepackage{hyperref}

  \newcommand{\calF}{\mathcal{F}}
  \newcommand{\calG}{\mathcal{G}}

  \newcommand{\calK}{\mathcal{K}}
  \newcommand{\calL}{\mathcal{L}}

  \newcommand{\calP}{\mathcal{P}}

  \newtheorem{theorem}{Theorem}[section]
  \newtheorem{proof of the main theorem}[theorem]{Proof of the Main Theorem}
  \newtheorem{proposition}[theorem]{Proposition}

  \newtheorem{corollary}[theorem]{Corollary}
  
  \newtheorem{lemma}[theorem]{Lemma}
  
  \newtheorem*{conjecture*}{Conjecture}

  \theoremstyle{definition}
  \newtheorem{definition}[theorem]{Definition}

  \newtheorem*{claim*}{Claim}

  \newtheorem{example}[theorem]{Example}
  
  \newtheorem*{question 1 *}{Question 1}
  \newtheorem*{question 2 *}{Question 2}
  \newtheorem*{question 3 *}{Question 3}
  \newtheorem*{answer*}{Answer}
  \newtheorem*{application*}{Application}
  \newtheorem*{ideas*}{ideas}
  \theoremstyle{remark}
  \newtheorem{remark}[theorem]{Remark}
  \newtheorem*{remark*}{Remark}
  \newtheorem*{theorem*}{Theorem}

  \newcommand{\param}{{\mathchoice{\mkern1mu\mbox{\raise2.2pt\hbox{$
  \centerdot$}}
  \mkern1mu}{\mkern1mu\mbox{\raise2.2pt\hbox{$\centerdot$}}\mkern1mu}{
  \mkern1.5mu\centerdot\mkern1.5mu}{\mkern1.5mu\centerdot\mkern1.5mu}}}
\renewcommand{\setminus}{{\smallsetminus}}

  \newcommand{\co}{\colon\thinspace}

   \newcommand{\F}{{\mathbb{F}}}

\newcommand{\out}{\ensuremath{\mathrm{Out}(\mathbb{F}) } }
\newcommand{\aut}{\ensuremath{\mathrm{Aut}(\mathbb{F}) }}
\renewcommand{\F}{\ensuremath{\mathbb{F} } }

\newtheorem*{fbcrh}{Theorem~\ref{fbcrh}}

\begin{document}


\title{Relative hyperbolicity of free extensions of free groups}

\author{Pritam Ghosh}\thanks{The first author was supported by the Ashoka University faculty research grant.}
 \address{Department of Mathematics \\
Ashoka University\\
  Haryana 131029, India\\}
  \email{pritam.ghosh@ashoka.edu.in}
\author {Funda G\"ultepe}\thanks{The second author was partially
supported by NSF grant DMS-2137611.}
\address{Department of Mathematics and Statistics\\
 University of Toledo\\
 Toledo, OHIO}
\email{funda.gultepe@utoledo.edu}

\begin{abstract}
We study relative hyperbolicity of free-by-free groups: Given a collection of  outer automorphisms of the free group we give sufficient conditions for some power of these automorphisms to generate a free group so that the corresponding  extension is relatively hyperbolic. Namely, if $\phi_1, \cdots \phi_k $  is a collection of exponentially growing outer automorphisms with a common invariant \emph{subgroup system} such that any conjugacy class in the complement of this system grows exponentially under iteration by all $\phi_i$, then such a subgroup system can be used to construct a collection of peripheral subgroups relative to which, the extension of $\F$ by the free group generated by sufficiently high powers of $\phi_1, \cdots \phi_k $  will be hyperbolic.
Moreover, we show that our conditions are also necessary if we are given a free-by-free group with a relative hyperbolic structure where the generating outer automorphisms have \emph{cusp-preserving} property (motivated by homeomorphisms of surfaces which preserve cusps). \\

  \vspace{0.5cm}

\end{abstract}

\maketitle

\section{Introduction}
\label{intro}

Let $\F$ be a free group of finite rank $\geq 3$. The quotient group
$\aut/\mathrm{Inn}(\F)$, denoted by $\out$,  is called the group  of outer automorphisms of $\F$. An element $\phi\in \out$ is called \emph{exponentially growing} if for some conjugacy class $[\alpha]$ of an element $\alpha\in \mathbb F$,  the word length of $\phi^{i}([\alpha])$ grows exponentially with $i$ for any fixed generating set of $\F$. We will call a subgroup of $\out$ exponentially growing if all of its elements are. For any subgroup ${Q}<\out$, the short exact sequence
 \[1\to \F \to \aut \to \out \to 1 \]
 induces a short exact sequence
 \[ 1\to \F \to E_{Q}  \to {Q} \to 1.\]
 The extension group $E_{Q}$, which is the pullback of ${Q}$ to $\aut$, is a subgroup of $\aut$. 
 In fact any extension, say ${E}$  of $\F$ always surjects onto some $E_{Q}$ in this form via the homomorphism $\pi:{E} / \F \to \out$. 
  When ${Q}$ is free group of rank $\geq 2$, we say that $ E_{Q}$ is a free-by-free group and it is known that $ E_{Q}\cong \F\rtimes \widehat{Q}$, where $\widehat{Q}$ is some (any) lift of ${Q}$ to $\aut$. 

 In this paper we give  sufficient conditions for $E_Q$ to be relatively hyperbolic using the dynamics of the generators of ${Q} < \out$. In addition, we show that the conditions are also necessary when we impose the additional condition that $Q$ preserves the conjugacy class of the subgroup which is the intersection of $\F$ with peripheral subgroups (we call this \emph{cusp preserving}) and high enough powers of automorphisms are taken.
 This contributes to a growing body of literature devoted to understanding the geometry of extensions of free groups. 

 Introduced by Gromov in \cite{Gro-87}, relative hyperbolicity, which is the hyperbolicity of a group \emph{relative} to a collection $\calP$ of \emph{peripheral} subgroups,  has been a major topic in the study of  geometry of groups. Relatively hyperbolic groups were first visible as fundamental groups of complete noncompact hyperbolic manifolds of finite volume. As such, relative hyperbolicity is not only a  generalization of hyperbolicity (where $\calP=\emptyset$), also of geometric finiteness in the Kleinian group setting (see \cite{BowGF, BowGF2,Bow-97}). 

 A group $\calG$ is hyperbolic relative to a collection $\calP$ of peripheral subgroups if and only if the (Cayley graph of) the group $\calG$ is hyperbolic after \emph{electrifying/coning off} the Cayley graph of  $\calP$, and the elements of $\calP$  satisfy  a
“bounded coset penetration” property (\cite{Fa-98}). Due to the fact that each emphasize a different feature of relative hyperbolicity, we will use the different definitions and characterizations of relative hyperbolicity due to Farb (\cite{Fa-98}), Osin (\cite{Osin-06}) and Bowditch (\cite{Bow-97}). Moreover, using a combination theorem of Mj--Reeves (\cite{{MjR-08}}) which characterizes relative hyperbolicity for the Cayley graph of the extension group, we will develop our own characterization of relative hyperbolicity for our setting.
 
 In this paper we investigate following questions when ${Q} < \out$ is a finite rank subgroup, in the most general sense:

\begin{question 1 *}
What are the necessary and sufficient conditions on $Q$ so that $E_Q$ is relatively hyperbolic ? 
\end{question 1 *}  
     \begin{question 2 *} How does one check these conditions and find ways to check for possible obstructions? 
     \end{question 2 *}
We address the sufficient direction from Question 1 in this paper by defining a new  set of peripheral subgroups (\emph{admissible subgroup systems}) for collections of exponentially growing elements of \out. Exponentially growing outer automorphisms of $\out$ are the most general types of outer automorphisms since they are characterized only by their growth rate; which implies they can be reducible, yet the class of exponentially growing outer automorphisms include the ones which show pseudo--Anosov like behavior: atoroidal and fully irreducible outer automorphisms. 

A \emph{subgroup system} of $\F$ is a finite collection of conjugacy classes of finite rank subgroups of $\F$. Every exponentially growing outer automorphism $\phi$ comes with an invariant set $\mathcal L^{\pm}(\phi)$ of attracting and repelling laminations for each corresponding exponentially growing stratum (see section \ref{sec:5}), which are computed using train tracks. We exploit this to determine a list of conditions on lamination sets (in section \ref{SA}) which prevent hyperbolicity, and we call a malnormal subgroup system which satisfies the properties we listed  an \emph{admissible subgroup system}. As such, admissible subgroup systems will act like building blocks of peripheral subgroups for us (which will be coned-off).

 Given an admissible  subgroup system $\mathcal{K}$, we let $\mathcal{L}_\mathcal{K}^\pm(\phi)$ denote the subset of  $\mathcal L^{\pm}(\phi)$ consisting of elements which are not supported by $\mathcal{K}$. We prove, 
\begin{theorem}\label{main1}
Let $\phi_1, \ldots , \phi_k $ be a collection of exponentially growing outer automorphisms of $\mathbb F$ such that $\phi_i, \phi_j$ do not have a common power whenever $i\neq j$. Then the following statements are equivalent:
\begin{enumerate}
    \item There exists a subgroup system $\mathcal{K} = \{[K_1], \ldots , [K_p]\}$ which is admissible for each $\phi_i$ and for which $\mathcal{L}_\mathcal{K}^\pm(\phi_i) \cap \mathcal{L}_\mathcal{K}^\pm(\phi_j) = \emptyset$ whenever $i\neq j$. 
    \item {For every free group $Q$ generated by sufficiently high powers of $\phi_i$'s, $\F\rtimes\widehat{Q}$ is hyperbolic relative to a collection $\{K_s \rtimes \widehat{Q}_s\}_{s= 1}^p$ where $\widehat{Q}_s$ is a lift of $Q$ that fixes $K_s$. }
\end{enumerate}
\end{theorem}
Throughout the paper, we will assume that our automorphisms are \emph{rotationless}; which is a technical condition which ensures that everything periodic is fixed (see Section \ref{sec:CT}). The other conditions that we state above are very natural in the sense that they generalize  some important theorems given  in \cite [Theorem 5.1]{BFH-97}, \cite[Theorem 5.2]{BFH-97}, \cite[Proposition 3.13]{Gh-23} for free groups and  \cite[Theorem 4.9]{MjR-08}, \cite[Proposition 6.2]{Bow-07} for surface groups with boundary.


We say that $\phi_i, \phi_j$, $i\neq j$ are \emph{independent relative to $\mathcal{K}$ } if generic leaves of  elements of $\mathcal{L}^\pm_\mathcal{K}(\phi_i)$ and  $\mathcal{L}^\pm_\mathcal{K}(\phi_j)$ are not asymptotic to each other (see Section \ref{MtoRH}).
Let $\phi_1, \ldots , \phi_k $ be a collection as in Theorem \ref{main1}. Then condition $(1)$ in Theorem \ref{main1} can be rephrased to indicate the existence of an admissible system for an independent set $\phi_1, \ldots , \phi_k $ of automorphisms relative to that admissible system.


\subsection{ From weak attraction (to the laminations) to admissible subgroup systems}
An outer automorphism is represented by a relative train track map on a marked metric graph; which is the analog of Thurston's normal form for surface diffeomorphisms (\cite{BH-92}). 
Moreover, the behavior of an automorphism on the strata  of a filtration (which is an analog of subsurface) of a marked graph will determine how the lines and circuits will converge. This dynamics will be one of our tools:  
It is crucial for relative hyperbolicity to determine the circuits and lines (of a marked graph) which are \emph{(weakly) attracted} to the attracting laminations. The ones which do get attracted are the ones which are likely to satisfy \emph{flaring conditions} needed to prove (relative) hyperbolicity and the ones which do not  must be trapped in some \emph{peripheral subgroups}. We collect the lines and circuits which do not get attracted to the laminations under \emph{non--attracting subgroup system}s and  these will be the ones to be \emph{coned off (electrified)} to obtain relative hyperbolicity. 

The definition of the \emph{admissible subgroup system}  $\mathcal{K}$ for a single element $\phi\in\out$ is based on a list of standing assumptions (\ref{K}) the laminations associated to $\phi$ need to satisfy to ensure the attraction to the laminations (equivalently, exponential growth under iterations). We then state a \emph{generalized weak attraction theorem} (Theorem \ref{modwat})  which says that
a conjugacy class is not weakly attracted to an element of $\mathcal{L}_\mathcal{K}^+(\phi)$ if and only if it is carried by $\mathcal{K}$.
Using this, we conclude that every conjugacy class that is not carried by $\mathcal{K}$ flares (relative to $\mathcal{K}$): 
we first generalize the notion of
\emph{legality} of circuits (Section \ref{sec:legrate}) which first appeared in \cite{BFH-97} to our setting. We achieve this by calculating the \emph{legality} for each exponentially growing stratum whose laminations are not carried by $\mathcal{K}$. This calculation later tells us what portions of our paths will survive (not backtrack) after electrifying. Generalized weak attraction theorem now says that when iterated sufficiently many times by either $\phi$ or $\phi^{-1}$  every
conjugacy class that is not carried by $\mathcal{K}$ 
gains sufficiently long legal segments  relative to $\mathcal{K}$ (Lemma \ref{legality}). This result is then used to prove ``conjugacy flaring'' in  Proposition \ref{conjflare} and ``strictly flaring'' condition in Proposition \ref{strictflare}. The technique of proof in both these flaring results is a generalization of the technique used in \cite{BFH-97}. 

An example of an admissible subgroup system is the \emph{nonattracting sink} of an automorphism, whose construction has been shown in this paper. Further properties of the nonattracting sink and how it effects (relative) hyperbolicity will be investigated  in a subsequent paper.

\subsection{From admissible subgroup systems to relative hyperbolicity}
To prove relative hyperbolicity, we electrify all the copies of the admissible subgroup system $\calK$ in the extension group; and use three \emph{flaring} conditions for conjugacy classes along pairs of paths on electrified Cayley graph of the extension group (Propositions \ref{conjflare}, \ref{strictflare},\ref{cbhfc}) to show that electrified graph flares. This results in relative hyperbolicity, as indicated in Theorem 4.6 of Mj--Reeves (\cite{{MjR-08}}) (recorded as Theorem \ref{mj-reeves} in Section \ref{sec:relhyp}). 

The electrification process has two steps, and the \textit{cone bounded hallways strictly flare condition} (Proposition \ref{cbhfc}) is what gives \emph{strong} relative hyperbolicity. A generalization of the Bestvina-Feighn Annuli Flare Condition (\cite{BF-92}) into the framework of relative hyperbolicity; cone bounded hallway strictly flare condition is what establishes the \emph{mutually cobounded} property of Bowditch (see \cite[Definition 2.3, Lemma 2.4]{MjR-08} and hence \emph{bounded coset penetration property}. 

To prove the flaring conditions we use all the tools of the CT-- maps (which are special type of relative track maps, see Section \ref{sec:CT}) and the weak attraction property, which were generalized from the techniques used in \cite{BH-92}, \cite{BFH-97}, \cite{BFH-00}, \cite{FH-11}, \cite{HM-09} and \cite{HM-20}.

Following theorem is an outcome of the new generalized flaring conditions,  which was proven by Ghosh \cite{Gh-23} with more strict conditions on the flaring.




\begin{theorem} \label{fbcrh} 
Let $\phi\in\out$ and $\mathcal{K}=\{[K_1], \ldots , [K_p]\}$ be an admissible subgroup system for $\phi$.

Then, the group $\F \rtimes \langle \phi \rangle$ is strongly hyperbolic relative to the collection of subgroups $\{K_s \rtimes \langle \Phi_s \rangle\}_{s=1}^p$, where $\Phi_s\in \aut$ is a lift of $\phi$ such that $\Phi_s (K_s) = K_s$.

\end{theorem}

Existence of relatively hyperbolic free-by-cyclic extensions were given also by Dahmani-Suraj in \cite{Dah-Relhypfbc} (in fact, they prove something more general), however, our theorem is constructive and this construction is what we eventually exploit to get to the general free-by-free case.

\subsection{From relative hyperbolicity to admissible subgroup systems}

To prove the existence of an admissible subgroup system $\mathcal K$ given in Theorem \ref{main1} we use  various definitions of relative hyperbolicity (\cite{Osin-06}, \cite{Fa-98}, \cite{MjR-17}) along with malnormality. In this direction we do not need the group $Q$ to be free (see Theorem \ref{necrelhyp}).  The key takeaway here is that existence of  an admissible subgroup system for $Q$ (an algebraic property) and cusp-preserving relative hyperbolicity of $E_Q$ (a geometric property) are almost interchangeable notions; modulo passing to high powers. 

\subsection{History and motivation}
The interest in the geometry of extensions of groups started with Thurston's hyperbolization theorem \cite{Tharxiv} which states, in algebraic terms, that the $\mathbb Z$--extension of $\pi_1(S)$, where $S$ is a closed surface, is hyperbolic if and only if the cyclic group $\mathbb Z$ is generated by a pseudo-Anosov element of the mapping class group of $S$.  In the case of extensions of free groups, relative hyperbolicity is more appropriate notion given the variety of dynamical behavior of elements of $\out$ as there is more than one type of \emph{pseudo-Anosov behavior} in the $\out$ setting.

First generalization of Thurston's result to the free group setting is due to Bestvina--Feighn. Using a \emph{combination theorem}  they find examples of hyperbolic free-by-cyclic groups (\cite{BF-92}). The question of hyperbolicity of free-by-cyclic groups was settled by Brinkmann who showed that any atoroidal
automorphism induces a hyperbolic free-by-cyclic group (\cite{Br-00}). Relative hyperbolicity of free-by-cyclic groups was addressed almost simultaneously by Ghosh \cite{Gh-23}, Dahmani-Li \cite{DahLi-20} and Dahmani-Krishna \cite{Dah-Relhypfbc} and the property of non-relative hyperbolicity in this setting was addressed by Hagen (\cite{HagNib-18}). It was shown by Ghosh \cite{Gh-23} and Hagen \cite{HagNib-18} that when $\mathcal{Q}$ is infinite cyclic, $E_Q$ is hyperbolic relative to some collection of peripheral subgroups if and only if every conjugacy class $[w]$ in $\F$ grows exponentially under iteration by $\phi$, provided that $w$ is not in some conjugate of a peripheral subgroup.
Sufficient conditions for hyperbolicity of extensions of free groups have been addressed by Ghosh \cite{Gh-23} and  Uyanik \cite{Uya-17} for the free-by-free case and Dowdall-Taylor  for a more general case \cite{DowTay1}. 
In contrast, only known works addressing the relative hyperbolicity of extensions of free groups generated by  non-cyclic subgroups, are from the theory of mapping class groups of puncture surfaces (see Bowditch \cite{Bow-07}, Mj-Reeves \cite{MjR-08}). We attempt to partly fill this gap.

\subsection{Plan of the paper:}

In Section \ref{Prelim} we give all the preliminaries and the tools we use.

In Section \ref{SRH} we define admissible subgroup system and state the notion of legality of paths in our setting. Flaring conditions (conjugacy flaring \ref{conjflare}, hallways flaring (\ref{strictflare}), cone-bounded hallways flaring (\ref{cbhfc}) stated in the propositions in this section depend on this legality calculation used in length comparison lemma \ref{comparison} and legality growth \ref{legality}. We also give a proof of the free-by-cyclic case, Theorem \ref{fbcrh} in this section.

In Section \ref{MtoRH} we prove the first direction of Theorem  \ref{main1}.  The construction of an admissible subgroup system is explained in details in section \ref{admissconst}.  Section \ref{RHtoM} is devoted to prove the other direction of Theorem \ref{main1}.

\subsection {Acknowledgements} 
We thank Spencer Dowdall, Mark Feighn, Chris Leininger, Lee Mosher, Pranab Sardar, and Sam Taylor for their remarks and comments on the earlier version of the paper. We also would like to thank the anonymous referee for carefully reading the paper and providing numerous suggestions that have improved the exposition.

 \section{Preliminaries}\label{Prelim}

\subsection{Marked graphs, circuits and path: }

 A \emph{core graph} is a graph $G$  with no valence 0 or 1 vertices. A \textit{marked graph} is a graph $G$ which is a core graph equipped with  a homotopy equivalence $m: G\to R_n$ to the rose $R_n$  (where $n = \text{rank}(\F)$). The fundamental group of $G$ therefore can be identified with $\F$ up to an inner automorphism. There is a natural $\F$-equivariant uniform quasi isometry between the universal covers of $R_n$ and $G$.
 A \textit{circuit} in a marked graph is an immersion (i.e. locally injective and continuous map)
 of $S^1$ into $G$. The set of circuits in $G$ can be identified
 with the set of conjugacy classes in $\F$.

A \emph{path} is an immersion of the interval $[0, 1]$ into $G$ with endpoints at vertices of $G$. Every path can be written as a finite concatenation of edges of $G$, so that two consecutive edges do not cancel out (i.e., without backtracking)

  A \emph{ray} $\gamma$ is ``one-sided'' infinite concatenation of edges $\gamma = E_i E_{i+1} \ldots $, $i\in \mathbb{Z}_{\geq 0}$ in $G$ without backtracking, \emph{i.e.} $E_{j-1}^{-1} \neq E_j\neq E_{j+1}^{-1}$ (where $E_{s}^{-1}$ is $E_s$ traveled in opposite direction).
 A \emph{line} $\ell$ is a bi-infinite concatenation $\ell = \ldots E_{i-1} E_i E_{i+1} \ldots$, $i\in \mathbb{Z}$ of edges of $G$ without backtracking. Two lines are said to be asymptotic if they have a common sub-ray.
  Any continuous map $f$ from $S^1$ or $[0,1]$ to $G$ is freely homotopic to a locally injective and continuous map, in other words can be \textit{tightened} to a circuit or path. We will denote the tightened image of a path $\alpha$ under $f$ by $f_{\#}(\alpha)$ and we will  not distinguish between circuits or paths that differ by a homeomorphism of their respective domains.
\subsection{Topological representative, EG strata, NEG strata:}
\label{sec:2}
 A \textit{filtration}  of a marked graph $G$ is a strictly increasing sequence $G_0 \subset G_1 \subset \cdots \subset G_k = G$ of subgraphs $G_r$ with no isolated vertices that are called \textit{filtration elements}. The filtration is \textit{$f$-invariant} (or a homotopy equivalence $f: G \to G$ \textit{respects} the filtration) if $f(G_r) \subset G_r$ for all $r$.

The subgraph $H_r = G_r \setminus G_{r-1}$ together with the  endpoints of edges of 
 $H_r$ is called the \textit{stratum of height $r$}. The \textit{height}\index{height} of subset of $G$ is the minimum $r$ such that the subset is contained in $G_r$.

  Given an $f$-invariant filtration, 
the \textit{transition matrix} $M_r$ of the stratum $H_r$ is the square matrix whose $j-{\text{th}}$ column records the number of times
the image of an edge in $H_r$ under $f$ intersects the other edges $H_r$ (counted in either direction).
	If for each  $i,j$, the $(i,j)-${\text{th}} entry of some power of $M_r$ is nonzero then $M_r$ is said to be irreducible. In this case we say that the associated stratum $H_r$ is also \textit{irreducible}. By the Perron-Frobenius theorem, when $H_r$ is irreducible the matrix $M_r$ has a unique eigenvalue $\lambda \ge 1$, called the \textit{Perron-Frobenius eigenvalue}, for which some associated eigenvector has positive entries.
if $\lambda=1$ then $H_r$ is a \textit{nonexponentially growing (NEG) stratum}  whereas if $\lambda>1$ then we say that $H_r$ is an \textit{exponentially growing (EG) stratum}.

 A \textit{topological representative} of an automorphism $\phi\in\out$ is a homotopy equivalence $f:G\rightarrow G$ that takes vertices to vertices and edges to edge-paths of a marked graph $G$ with marking $\rho: R_n \rightarrow G$. A nontrivial path $\alpha$  in G is a \emph{periodic Nielsen path} if $f^k_{\#}(\alpha)=\alpha$ for some $k$, where the smallest such $k$ is the \emph{period}. $\alpha$ is a Nielsen path if $k=1$. A
periodic Nielsen path is \emph{indivisible} (iNP) if it cannot be written as a concatenation of 
nontrivial periodic Nielsen paths.

Given a topological representative $f: G\to G$  one can define a new map $Tf$ by setting $Tf(E)$ to be the first edge in the edge path associated to $f(E)$. For a \emph{turn}, which is a pair $\{E_i, E_j\}$ of edges with the same initial vertex, we let $Tf(E_i,E_j) = (Tf(E_i),Tf(E_j))$. Hence, $Tf$ is a map that takes turns to turns. We say that a
non-degenerate (i.e, $i\neq j$) turn is \emph{illegal} if for some iterate of $Tf$ the turn becomes degenerate; otherwise the
 turn is legal. A path is said to be a \emph{legal path} if it contains only legal turns and it is $r-legal$ if it is of height $r$ and all of its illegal turns are in $G_{r-1}$.

 \textbf{Relative train track map.} 
 Given $\phi\in \out$ and a topological representative $f:G\rightarrow G$ with a filtration $G_0\subset G_1\subset \cdot\cdot\cdot\subset G_k$ which is preserved by $f$,
 we say that $f$ is a relative train track map if the following conditions are satisfied: \label{rtt}
 \begin{enumerate}
  \item $f$ maps $r$-legal paths to  $r$-legal paths.
  \item If $\gamma$ is a path in $G_{r-1}$, with endpoints in $G_{r-1}\cap H_r$   then $f_\#(\gamma)$ is non-trivial.
  
  \item If $E$ is an edge in $H_r$ then $Tf(E)$ is an edge in $H_r$. In particular,
  every turn consisting of a direction of height $r$ and one of height $< r$  is legal.
 \end{enumerate}
 
By \cite[Theorem 5.1.5]{BFH-00}  every  $\phi\in\out$ has a relative train-track map representative which satisfies some useful conditions in addition to the ones we listed above. 

\subsection{Weak topology}
\label{sec:3}
Given a graph $G$, there is an equivalence relation on the set of all paths, rays, lines, circuits in $G$. Namely, two of them are equivalent if they differ by a homeomorphism of their domains. 
 Let $\widehat{\mathcal{B}}(G)$ denote the compact space of equivalence classes of circuits and paths (finite paths, rays and lines) in $G$ whose endpoints (if any) are vertices of $G$.  
For each finite path $\gamma$ in $G$, set of all paths and circuits in $\widehat{\mathcal{B}}(G)$ which have $\gamma$ as its subpath is denoted by $\widehat{N}(G,\gamma)$ and the collection of all such sets gives a basis for a topology on $\widehat{\mathcal{B}}(G)$ 
 called \textit{weak topology}. Let $\mathcal{B}(G)\subset \widehat{\mathcal{B}}(G)$ be the compact subspace of all lines in $G$ with the induced topology. 

 Up to reversing the direction, two distinct points in $\partial \mathbb{F}$ determine a \emph{line} completely since there is only one line that joins these two points. Let $\widetilde{\mathcal{B}}=\{ \partial \mathbb{F} \times \partial \mathbb{F} - \vartriangle \}/(\mathbb{Z}/2 \mathbb Z)$ be the set of pairs of boundary points of $\F$ where $\vartriangle$ is the diagonal and $\mathbb{Z}/2 \mathbb Z$ acts by interchanging factors. We can give the weak topology to
$\widetilde{\mathcal{B}}$, induced by the Cantor topology on $\partial \mathbb{F}$.

$\mathbb{F}$ acts on $\widetilde{\mathcal{B}}$ with a compact but non-Hausdorff quotient space $\mathcal{B}=\widetilde{\mathcal{B}}/\mathbb{F}$, hence the topology is called  weak. The quotient topology is also called the \textit{weak topology}. For any marked graph $G$, there is a natural identification $\mathcal{B}\approx \mathcal{B}(G)$.

   Elements of $\mathcal{B}$ are also called \textit{lines}. A lift of a line $\gamma \in \mathcal{B}$ is an element  $\widetilde{\gamma}\in \widetilde{\mathcal{B}}$ that projects to $\gamma$ under the quotient map and the two elements of $\widetilde{\gamma}$ are called its endpoints.

For any $\pi_1$-injective map $f: G\to G'$ we can define a natural induced map $f_\# : \widehat{\mathcal{B}}(G)\to \widehat{\mathcal{B}}(G')$. If $\gamma$ is a finite path or a circuit then $f_\#(\gamma)$ is the unique path in $G'$ that is homotopic relative to the endpoints to  (or, tightened to) the $f$-image of $\gamma$. For a line $\gamma$, we pass via the universal covers of $G, G'$ and use Gromov boundaries to define the element $f_\#(\gamma)$ as the projection of the unique element of $\widehat{\mathcal{B}}(G')$ with same endpoints as $f(\gamma)$ which is also a continuous map.

When $f$ is a homotopy equivalence, $f_\#$ is a homeomorphism.  If $f$ is a topological representative, then this homeomorphism is the identity map.

  A line (or a path) $\gamma$ is said to be \textit{weakly attracted} to a line (path) $\beta$ under the action of $\phi\in\out$, if  for some  $k$, $\phi^k(\gamma)$ converges to $\beta$ in the weak topology, in  other words, if any given finite subpath of $\beta$ is contained in $\phi^k(\gamma)$ for some  $k$. Similarly if we have a homotopy equivalence $f:G\rightarrow G$,  a line(path) $\gamma$ is said to be \textit{weakly attracted} to a line(path) $\beta$ under the action of $f_{\#}$ if the $f_{\#}^k(\gamma)$ weakly converges to $\beta$.

The set of lines  $l\in \mathcal{B}(G)$ that are elements of the weak closure of a ray $\gamma$ in $G$ is called the \textit{accumulation set} of  $\gamma$. This is the set of lines $l$ such that every finite subpath of $l$
occurs infinitely many times as a subpath $\gamma$. Similarly, the weak accumulation set of some point  $\xi\in\partial\mathbb{F}$ is the set of lines in the weak closure of any of the asymptotic rays in its equivalence class.	

\subsection{Free factor systems and malnormal subgroup systems.}
A  subgroup $K<\F$ is \emph{malnormal} if $xK_sx^{-1}\cap K$ is trivial for all $x\in \F-K$. A finite collection
$\mathcal{K} = \{[K_1], [K_2], .... ,[K_s]\}$ of conjugacy classes of nontrivial finite rank subgroups $K_s<\F$ is called a \textit{subgroup system}. Define a subgroup system to be \textit{malnormal} if each of its subgroups is malnormal and  for all $[K_s],[K_t]\in \mathcal{K}$ if some representatives of $ [K_s]$ and $[K_t]$ intersect non-trivially then $s=t$.
Given two malnormal subgroup systems $\mathcal{K}, \mathcal{K}'$ we give a partial ordering  by defining $\mathcal{K}\sqsubset \mathcal{K}'$ if for each conjugacy class of subgroup $[K]\in \mathcal{K}$ there exists some conjugacy class of subgroup $[K']\in \mathcal{K}'$ such that $K < K'$.
	
 Given a finite collection $\{K_1, K_2,.....,K_s\}$ of subgroups of $\F$ , we say that this collection determines a \textit{free factorization} of $\F$ if $\F$ is the free product of these subgroups, that is,
$\F = K_1 * K_2 * .....* K_s$.
A \emph{free factor system} is a finite collection of conjugacy classes of subgroups $\mathcal{F}:=\{[F_1], [F_2],.... [F_p]\}$ of $\F$ such that there is a free factorization of $\F$ of the form
$\F = F_1 * F_2 * ....*F_p* B$, where $B$ is some (possibly trivial) finite rank subgroup of $\F$ . Every free factor system is a malnormal subgroup system.

 A malnormal subgroup system $\mathcal{K}$ \textit{carries a conjugacy class} $[c]\in \F$ if there exists some $[K]\in\mathcal{K}$ such that $c\in K$. Also, we say that $\mathcal{K}$ carries a line $\gamma$ if one of the following equivalent conditions hold:
	\begin{enumerate}
 		\item $\gamma$ is the weak limit of a sequence of conjugacy classes carried by $\mathcal{K}$.
 		\item There exists some $[K]\in \mathcal{K}$ and a lift $\widetilde{\gamma}$ of $\gamma$ so that the endpoints of 		$\widetilde{\gamma}$ are in $\partial K$.
	\end{enumerate}
	 \begin{lemma}\cite[Fact 1.8]{HM-20}
	 For each subgroup system $\mathcal{K}$ the set of lines / circuits carried by $\mathcal{K}$ is a closed set in the weak topology.
\end{lemma}

 For any marked graph $G$ and any subgraph $H \subset G$, the fundamental groups of the noncontractible components of $H$ form a free factor system, denoted by $[\pi_1(H)]$. Every free factor system $\mathcal{F}$ can be realized as $[\pi_1(H)]$ for some nontrivial core subgraph H of some marked graph $G$. An equivalent way of saying that a line or circuit  $\gamma$ is carried by $\mathcal{F}$ is that for some marked graph $G$ and a subgraph $H \subset G $ with $[\pi_1(H)]=\mathcal{F}$, the realization of $\gamma$ in $G$ is contained in $H$.

 Given two free-factor systems there is a natural operation one can perform  called `` meet''. It is an extension of the notion of ``intersection '' of subgroups, extended to free factor systems.
	\begin{lemma}
 [\cite{BFH-00}, Section 2.6] Every collection $\{\mathcal{F}_i\}$ of free factor systems has a well-defined meet $\wedge\{\mathcal{F}_i\} $, which is the unique maximal free factor system $\mathcal{F}$ such that $\mathcal{F}\sqsubset \mathcal{F}_i$ for all $i$. Moreover,
 for any free factor $F< \F$ we have $[F]\in \wedge\{\mathcal{F}_i\}$ if and only if there exists an indexed collection of subgroups $\{A_i\}_{i\in I}$ such that $[A_i]\in \mathcal{F}_i$ for each $i$ and $F=\bigcap_{i\in I} A_i$.
\end{lemma}

The \textit{free factor support}  $\mathcal{F}_{supp}(B)$ of a set of lines $B$ in $\mathcal{B}$ is defined as the meet of all free factor systems that carries $B$ (\cite{BFH-00}). If $B$ is a single line then $\mathcal{F}_{supp}(B)$ is
single free factor. We say that a set of lines, $B$, is \textit{filling} if $\mathcal{F}_{supp}(B)=[\F]$

  Given a malnormal subgroup system $\mathcal{K}$ and a sequence of lines / circuits $\{\gamma_n\}$, if every weak limit of every subsequence of $\{\gamma_n\}$ is carried by $\mathcal{K}$, then $\gamma_n$ is carried by $\mathcal{K}$ for all sufficiently large $n$ (\cite[Lemma 1.11]{HM-20}) .


\subsection{Attracting laminations and nonattracting subgroup systems:}
\label{sec:5}

	For any marked graph $G$, the natural identification $\mathcal{B}\approx \mathcal{B}(G)$ induces a bijection between the closed subsets of $\mathcal{B}$ and the closed subsets of $\mathcal{B}(G)$. A closed
subset of any of these two sets is called a \textit{lamination}, denoted by $\Lambda$. Given a lamination $\Lambda\subset \mathcal{B}$ we look at the corresponding lamination in $\mathcal{B}(G)$ as the
realization of $\Lambda$ in $G$. An element $\lambda\in \Lambda$ is called a \textit{leaf} of the lamination.
	
 A lamination $\Lambda$ is called an \textit{attracting lamination} for $\phi$ if it is the weak closure of a line $\ell$ (called the \textit{generic leaf of $\Lambda$}) satisfying the following conditions:
\begin{itemize}
 \item $\ell$ is bi-recurrent leaf of $\Lambda$. This means that every finite subpath of $\ell$ occurs infinitely many times as subpath in both directions.
\item $\ell$ has an \textit{attracting neighborhood} $V^+$, in the weak topology, with the property that every line in $V^+$ is weakly attracted to $\ell$. (see \cite[Definition 3.1.1]{BFH-00})
\item no lift $\widetilde{\ell}\in \mathcal{B}$ of $\ell$ is the axis of a generator of a rank 1 free factor of $\F$ .
\end{itemize}
	 Attracting neighborhoods of $\Lambda$ are defined by choosing sufficiently long segments of generic leaves of $\Lambda$. Note that if $V^+$ is an attracting neighborhood, then $\phi(V^+)\subset V^+$. 
	
  By \cite{BFH-00},  associated to each $\phi\in \out$ is a finite set $\mathcal{L}^+(\phi)$ of laminations,  called the set of \textit{attracting laminations} of $\phi$. Similarly we define  the set of attracting laminations (or the set of \emph{repelling laminations}) $\mathcal{L}^-(\phi)$ of $\phi^{-1}$ (or of $\phi$). Both $\mathcal{L}^+(\phi)$ and $\mathcal{L}^-(\phi)$ are $\phi$-invariant (\cite[Lemma 3.1.13, Lemma 3.1.6]{BFH-00}). If $f: G\to G$ is a relative train-track map representing $\phi$, then there is a bijection between  the set $\mathcal{L}^+(\phi)$ and the set of exponentially growing strata in $G$ and this bijection is determined by the height of a generic leaf of $\Lambda^+\in\mathcal{L}^+(\phi)$ (\cite[Section 3]{BFH-00}).

 Free factor support of an element of $\mathcal{L}^+(\phi)$ or $\mathcal{L}^-(\phi)$ is defined by a single free factor $\{[F]\}$. An element of $\mathcal{L}^+(\phi)$ and an element of $\mathcal{L}^-(\phi)$ are dual if they have the same free factor support. This relation imposes a bijection between these two sets \cite[Lemma 3.1.2]{BFH-00}. Free factor support of $\Lambda^+\in\mathcal{L}^+(\phi)$ is equal to free factor support of a generic leaf of $\Lambda^+$ (\cite[Corollary 2.6.5, Corollary 3.1.11]{BFH-00}). If $\Lambda_1, \Lambda_2 \in \mathcal{L}^+(\phi)$ then $\Lambda_1 = \Lambda_2 \Leftrightarrow  \mathcal{F}_{supp}(\Lambda_1) = \mathcal{F}_{supp}(\Lambda_2)$ (\cite[Fact 1.14]{HM-20}).

 A line/circuit $\gamma$ is said to be \emph{weakly attracted} to $\Lambda_1\in\mathcal{L}^+(\phi)$ if $\gamma$ is weakly attracted to some (hence every) generic leaf of $\Lambda$ under the action of $\phi$. No leaf of any element of $\mathcal{L}^-(\phi)$ is ever attracted to any leaf of any element of $\mathcal{L}^+(\phi)$. No Nielsen path is weakly attracted to either an element of $\mathcal{L}^+(\phi)$ or an element of $\mathcal{L}^-(\phi)$.

 If $\Lambda^+$ is an attracting lamination and $\Lambda^-$ is a repelling lamination of $\phi$, then they cannot have leaves which are asymptotic, otherwise it would violate the fact that no leaf of $\Lambda^-$ is weakly attracted to $\Lambda^+$. This allows us to choose sufficiently long subpaths of generic leaves of $\Lambda^+, \Lambda^-$ and construct attracting and repelling neighborhoods $V^+, V^-$ so that $V^+\cap V^-=\emptyset$.
\begin{itemize}	
\item An element of $\Lambda\in \mathcal{L}^+(\phi)$ is said to be \emph{topmost} if there does not exists any $\Lambda_j\in\mathcal{L}^+(\phi)$ such that $\Lambda \subset \Lambda_j$. We say $\Lambda$ is \emph{bottommost} if there does not exist any $\Lambda_j\in\mathcal{L}^+(\phi)$ such that $\Lambda_j\subset \Lambda$. If $\Lambda$ is topmost or bottommost, then its dual also has the same property by Lemma 2.16 in \cite{HM-20}.
	\item Elements of $\mathcal{L}^+(\phi)$ can be divided into two distinct classes : \emph{geometric} laminations and \emph{nongeometric} laminations. An element $\Lambda^+\in \mathcal{L}^+(\phi)$ is said to be \emph{geometric} if there exists a finite, $\phi$-invariant collection of  distinct, nontrivial  conjugacy classes $\mathcal{C} = \{[c_1], [c_2], \ldots [c_s]\}$ such that $\mathcal{F}_{supp}(\mathcal{C}) = \mathcal{F}_{supp}(\Lambda^+)$ (\cite[Definition 2.19]{HM-20}). It is said to be \emph{nongeometric} otherwise.
\end{itemize}

\subsection{Nonattracting subgroup system}
\label{sec:6}

The \emph{nonattracting subgroup system} of an attracting lamination was first introduced by Bestvina-Feighn-Handel in \cite{BFH-00} and later explored in great detail by Handel-Mosher in \cite[Part III, pp-192-202]{HM-20}. Nonattracting subgroup system records the information about the lines and circuits which are not attracted to the lamination.  We will list some of the properties which are central to our proofs.
\begin{lemma}(\cite{HM-20}- Lemma 1.5, 1.6)
\label{NAS} Let $\phi\in\out$ and $\Lambda^+\in \mathcal{L}^+(\phi)$ be an attracting lamination such that $\phi(\Lambda^+) = \Lambda^+$. Then there exists a subgroup system $\mathcal{A}_{na}(\Lambda^+)$ such that:
 \begin{enumerate}
  \item $\mathcal{A}_{na}(\Lambda^+)$ is a malnormal subgroup system and the set of lines carried by $\mathcal{A}_{na}(\Lambda^+)$ is closed in the weak topology.
  \item A conjugacy class $[c]$ is not attracted to $\Lambda^+$ if and only if it is carried by $\mathcal{A}_{na}(\Lambda^+)$.
  \item  If $\Lambda^-$ and  $\Lambda^+$ are dual to each other, we have $\mathcal{A}_{na}(\Lambda^+)= \mathcal{A}_{na}(\Lambda^-)$.
  \item $\mathcal{A}_{na}(\Lambda^+)$ is a free factor system if and only if $\Lambda^+$ is not geometric.
  \item If $\{\gamma_n\}_{n\in\mathbb{N}}$ is a sequence of lines such that every weak limit of every subsequence of $\{\gamma_n\}$ is carried by $\mathcal{A}_{na}(\Lambda^+)$ then for all sufficiently large $n$, $\{\gamma_n\}$ is carried by $\mathcal{A}_{na}(\Lambda^+)$.  
  \item If each conjugacy class of  a finite rank subgroup $B < \F$ is not weakly attracted to $\Lambda^+_\phi$, then there exists some
  $A < \F$ such that $B< A$ and $[A]\in \mathcal{A}_{na}(\Lambda^+_\phi)$.
 \end{enumerate}
\end{lemma}
A fully irreducible automorphism which is induced by a pseudo-Anosov homeomorphism on a surface with non-empty boundary is called a \emph{geometric} fully irreducible automorphism.  Geometricity was relativized by Bestvina-Feighn-Handel \cite{BFH-00} by extending it to a property of an EG stratum  of a relative train track map.
The definition of geometricity is expressed in terms of the existence of
a ``geometric model"(See \cite{HM-20} for details of constructions of geometric models.)

When $\Lambda^+$ is geometric, the geometric model gives us the structure of the nonattracting subgroup system. Suppose $\Lambda^+$ has height $r$ and let $H_r$ be the EG strata associated to $\Lambda^+$. To be more precise, by the work done in the geometric model, we have that:
		\begin{enumerate}
			\item There is a unique closed indivisible Nielsen path $\rho_r$ of height $r$. (\cite[Fact 1.42]{HM-20})
			\item There is a compact, connected hyperbolic surface with totally geodesic boundary components $\partial_0 S, \partial_1 S, \ldots, \partial_m S$ and the conjugacy class determined by $\partial_0 S$ is represented by the circuit $\rho_r$ (\cite[Lemma 2.5, pp-74]{HM-20}).
			\item The conjugacy classes determined by $\partial_1 S, \ldots, \partial_m S$ are all carried by $[\pi_1 G_{r-1}]$ (\cite[Definition 2.1, Lemma 2.5]{HM-20}).
                \item The circuit in $G$ representing the conjugacy class corresponding to $\partial_0 S$ is $\rho_r$, and the free-factor support of $\rho_r$ is invariant under $\phi$ (\cite[Lemma 2.5]{HM-20}).
                \item $[\pi_1 G_{r-1}]$ is a free-factor system carried by the nonattracting subgroup system $\mathcal{A}_{na}(\Lambda^+)$.
		\end{enumerate}		
	 The facts listed above show that when $\Lambda^+$ is geometric, $\mathcal{A}_{na}(\Lambda^+) = \mathcal{F} \cup\{[F_m]\}$, where $\mathcal{F}$ is some free-factor system such that $[\pi_1 G_{r-1}] \sqsubset\mathcal{F}$. Moreover,  $\phi$ restricted to $F_m$ has polynomial growth and if $[c]$ is the conjugacy class determined by $\rho_r$, then $[c]$ is carried by $[F_m]$. Note that $[F_m]$ can never be the conjugacy class of a free factor. In case the associated strata is a top strata, $\mathcal{A}_{na}(\Lambda^+) = \mathcal{F} \cup\{[\langle c \rangle]\}$.

\subsection{Completely split improved relative train track (CT) maps and rotationless outer automorphisms}\label{sec:CT}
 Let  $f:G\rightarrow G$ be a relative train track map representing $\phi\in\out$. A \textit{splitting} of a line, path or a circuit $\alpha$ is a concatenation $ \ldots\alpha_0\alpha_1 \ldots\alpha_k\ldots $ of subpaths of $\alpha$ in $G$ 
 such that for all $i\geq 1$, we have  $f^i_\#(\alpha) =  \ldots f^i_\#(\alpha_0)f^i_\#(\alpha_1)\ldots f^i_\#(\alpha_k)\ldots$. 
 The subpath $\alpha_i'$s in a splitting are called \emph{terms} or \emph{components} of the splitting of $\alpha$. The notation $\alpha\cdot\beta$ will denote a splitting and $\alpha\beta$ will denote a concatenation of nontrivial paths $\alpha, \beta$.


 \textbf{Complete splittings:}  A splitting of a path or circuit $\alpha = \alpha_1\cdot\alpha_2\cdots\cdot \alpha_k$ is called  \textit{complete splitting} if for each component $\alpha_i$ one of the following is satisfied:
 \begin{itemize}
  \item $\alpha_i$ is an edge in some irreducible stratum.
  \item $\alpha_i$ is an indivisible Nielsen path.
  \item $\alpha_i$ is an exceptional path (\emph{i.e.} a path of the form $E_1 w^n E_2^{-1}$, where $w$ is a circuit which is not a power of some other circuit and $w$ is fixed by $f$; $f(E_1)= E_1 w^{d_1}, f(E_2)= E_2 w^{d_2}$ where $d_1 \neq d_2$ have the same sign).
  \item $\alpha_i$ is a maximal subpath of $\alpha$ in a zero stratum $H_i$ and $\alpha_i$ is taken (see \cite[Definition 4.4]{FH-11}).
 \end{itemize}

 \textbf{CT maps:} Completely split improved relative train track (CT) maps are topological representatives with particularly nice properties. CT maps are guaranteed to exist for \textit{rotationless} (see Definition 3.13 \cite{FH-11}) outer automorphisms, as has been shown in the following Theorem from \cite{FH-11}(Theorem 4.28).\\ 

Let $\phi\in\out$ and $f:G\to G$ be some topological representative. Consider a nested sequence of $\phi-$invariant free factor systems $\mathcal{F}_1 \sqsubset \mathcal{F}_2 \sqsubset \ldots \mathcal{F}_n $. We say that $f$ realizes this nested sequence of free systems if there exists a filtration $G_1\subset G_2 \subset \ldots G_p = G$ of subgraphs of $G$ such that for each $\mathcal{F}_i$ there exists $G_{p_i}$ so that $G_{p_i}$ is a core subgraph and  $[\pi_1(G_{p_i})] = \mathcal{F}_{i}$, where $1\leq i \leq n$ and $1\leq p_i \leq p$.
 \begin{lemma}
  For each rotationless $\phi\in \out$ and each nested sequence $\mathcal{F}_1 \sqsubset \mathcal{F}_2 \sqsubset \ldots \mathcal{F}_n $ of $\phi$-invariant free factor systems, there exists a CT map $f:G\rightarrow G$ that is a topological
  representative for $\phi$ that realizes $\calF$. 
 \end{lemma}

Feighn-Handel \cite{FH-11} showed that there exists some $k>0$ such that, given any $\phi\in\out$, $\phi^k$ is rotationless. So given any outer automorphism $\phi$, some finite power of $\phi$ has a completely split improved relative train track representative.

Recall that if the transition matrix $M_s$ corresponding to a stratum is the zero matrix (irreducible matrix)
then we say that that stratum $H_k$ is a zero (irreducible) stratum.

 The results in the following definition are some properties of CT's defined in the  work of Feighn-Handel in \cite{FH-11}.
 We will state only the ones we need here from \cite[Definition 1.29]{HM-20}.
\begin{lemma}
    
Given a relative train track map $f:G\rightarrow G$ with an invariant filtration $\emptyset = G_0 \subset G_1 \subset \cdots G_s = G$, which is a CT map representing $\phi$, the following properties are satisfied:
 \begin{enumerate}\label{CT}
  \item \textbf{(Completely split)} In each irreducible stratum, for each edge $E$, the path $f(E)$ is completely split.
  \item \textbf{(Vertices)} The endpoints of all indivisible Nielsen paths are vertices. The terminal endpoint of each nonfixed NEG edge is fixed and every periodic direction at such a vertex is fixed.
  \item \textbf{(Periodic edges)} Each periodic edge is fixed.
  \item \textbf{(Zero strata)} Each zero strata $H_i$ is contractible and enveloped by (i.e including an irreducible stratum and included in an EG stratum) a EG strata
  $H_s, s>i$, such that every edge of $H_i$ is a taken in $H_s$. Each vertex of $H_i$ is contained in $H_s$ and
  link of each vertex in $H_i$ is contained in $H_i \cup H_s$.
 
  
  \item \textbf{(Linear edges)} For each linear edge $E_i$ there exists a root free  Nielsen  path $w_i$ (i.e, if $w_i=b^k$ then $k=\pm 1$) such that $f_\#(E_i) = E_i w^{d_i}_i$ for some $d_i \neq 0$.
  
   \end{enumerate}

\end{lemma}
\subsection{Weak attraction theorem }
\label{sec:7}
\begin{lemma}[\cite{HM-20} Corollary 2.17]
\label{WAT}
 Let $\phi\in \out$ be exponentially growing. Let $\Lambda^\pm_\phi$ be a dual lamination pair for $\phi$. Then for any line $\gamma\in\mathcal{B}$ not carried by $\mathcal{A}_{na}(\Lambda^{\pm}_\phi)$ at least one of the following hold:
\begin{enumerate}
 \item $\gamma$ is attracted to $\Lambda^+_\phi$ under iterations of $\phi$.
   \item $\gamma$ is attracted to $\Lambda^-_\phi$ under iterations of $\phi^{-1}$.
\end{enumerate}
Moreover, if $V^+_\phi$ and $V^-_\phi$ are attracting neighborhoods for the laminations $\Lambda^+_\phi$ and $\Lambda^-_\phi$ respectively, there exists an integer $l\geq0$ such that for any line $\gamma\in\mathcal{B}$ at least one of the following holds:
\begin{itemize}
 \item $\gamma\in V^-_\phi$.
\item $\phi^l(\gamma)\in V^+_\phi$
\item $\gamma$ is carried by $\mathcal{A}_{na}(\Lambda^{\pm}_\phi)$.
\end{itemize}

\end{lemma}


 \subsection{Relative Hyperbolicity and Mj--Reeves Theorem}\label{sec:relhyp}

In this section we will describe the  setting we  will be working in and adapt  Mj--Reeves criteria of relative hyperbolicity into our setting. 
 Given a group $G$ and a collection of subgroups $K_{\alpha} < G$, the \emph{coned-off Cayley graph of $G$} or
  the \emph{electric space of $G$} relative to the collection $\{K_\alpha\}$ is a metric space which consists of the
  Cayley graph of $G$ and a collection of vertices $v_\alpha$ (one for each left coset of $K_\alpha$) such that each point of (a left coset of) 
  $K_\alpha$ is joined to (or coned-off at) $v_\alpha$ by an edge of length $1/2$. The resulting metric space is
  denoted by $(\widehat{G}, {|\cdot|}_{el})$.

 A group $G$ is said to be (weakly) relatively hyperbolic relative to the collection of subgroups
  $\{K_\alpha\}$ if $\widehat{G}$ is a $\delta-$hyperbolic metric space, in the sense of Gromov.
  $G$ is said to be strongly hyperbolic relative to the collection $\{K_\alpha\}$ if the coned-off
  space $\widehat{G}$ is weakly hyperbolic relative to $\{K_\alpha\}$ and it satisfies the
  \textit{bounded coset penetration} property (see \cite{Fa-98}): 
 for all $\calL>1$, there is a number $B(\calL)$ such that for an efficient (i.e passes through a cone point at most once) $\calL$--quasi geodesic $p$ in $\widehat{G}$ with $p(s)=v_{\alpha}$ we have:
\begin{itemize}
\item If $d_{K_\alpha}(p(s-1), p(s+1))\geq B(\calL)$, every efficient $\calL$--quasi geodesic $p'$ in $\widehat{G}$ with the same endpoints also pass through $v_{\alpha}$,
\item In that case, assuming $p'(s')=v_{\alpha}$,
\[d_{K_\alpha}(p(s-1), p'(s'-1))\leq B(\calL)\,\,\,\text{and}\,\,\, d_{K_\alpha}(p(s+1), p'(s'+1))\leq B(\calL) \]
\end{itemize}
In this paper, when we say ``relative hyperbolicity" we always mean ``strong relative hyperbolicity".

To develop the theorem we will be using for relative hyperbolicity, we need to translate some of the concepts into our setting. Let

 \[ 1\to \F \to  E_{Q} \to Q \to 1\]
be the short exact sequence that gives rise to the extension $ E_{Q}= \F \rtimes Q$.
The Cayley graph of $ E_{Q}$, which we will  denote by $\Gamma$ is a tree ($T$) of spaces whose edge groups and vertex groups are identified
with cosets of $\F$. When we refer to the underlying tree, we will use the notation $T$. In our case, the maps from the edge groups  to the vertex groups are all quasiisometries.

Let $\mathcal{K}=\{K_s\} < \F$ be a malnormal collection of subgroups which are fixed by each $Q_s\in \mathcal Q$ and whose properties we will  describe in Section \ref{SA}. We will first electrify all  cosets of $\{K_s\}$ in all $\F$ (in all vertex and edge groups) and obtain the coned off Cayley graph $(\widehat{\Gamma}, {|\cdot|}_{el})$ whose edge groups $\widehat\Gamma_e$ and vertex groups $\widehat\Gamma_v$ are all identified with $\widehat \F$ where the latter denotes the electrified $\F$ with along $\mathcal{K}$. Since $\F$ is hyperbolic and the collection of subgroups $\{K_s\}$  is  mutually malnormal and \textit{quasiconvex} in $\F$, $\widehat{\Gamma}$ is a tree of strongly relatively hyperbolic spaces. In the terminology of Mj--Reeves Theorem \ref{mj-reeves} below, $\widehat{\Gamma}$ is the \emph{induced tree of coned-off spaces} obtained from coning off $\Gamma $ along all copies of $K_s\in \F $ on which all the flaring conditions will be checked.  

We first give the definitions we will need.

\begin{definition} We call a disk $f\co [-m,m]\longrightarrow \widehat \Gamma$ a \emph{hallway} of length $2m$ if,
\begin{enumerate}
\item $f^{-1}(\bigcup \widehat\Gamma_{v}:v\in T)=\{-m, \cdots, m\}\times I $
\item $f$  maps $i\times I $ into a geodesic in a coned-off  vertex space $\widehat\Gamma_v$
\item $f$ is transverse to $\bigcup \widehat\Gamma_e$
\end{enumerate}

A hallway is \emph{$\lambda$-  hyperbolic} if,
\[\lambda \ell (f(0\times I)) \leq \max \{ \ell(f(-m\times I)), \ell(f(m\times I)) \}\]

A hallway is \emph{essential} if the edge path in $T$ resulting from projecting $\widehat \Gamma$
onto $T$ does not backtrack (and hence is a geodesic segment in the
tree $T$)
\end{definition}
 The \emph{cone locus} is the graph (forest ) whose vertices are cone points in the  vertex groups and whose edges are cone points in the edge groups. Connected components are called \emph{maximal cone subtrees}. 
In our case, maximal cone subtrees are $T_{s}$’s which
are the cosets of $\Gamma_{s}={K_s} \rtimes Q_s$ (after electrocuting the cosets of $K_{s}$ in $\Gamma_{s}$) inside $\Gamma$. 

\begin{definition}(Cone bounded hallway)
An essential hallway of
length $2m$ is cone-bounded if $f(i\times \partial I)$ lies in the cone locus for $i=\{-m,\cdots, m\}$.

\end{definition}

The main relative hyperbolicity result we will be using is the following:
\begin{theorem}[Theorem 4.6,\cite{ MjR-08}]\label{mj-reeves}  Let $\mathcal G$ be a finite graph of strongly relatively hyperbolic groups satisfying,
\begin{enumerate}
\item The quasi-isometrically (q.i) embedded condition,
\item Strictly type preserving condition,
\item The q.i--preserving electrocution condition,

and moreover, the induced tree of coned-off spaces satisfying,

    \item Hallways flare condition,
    \item Cone-bounded hallways strictly flare condition
 \end{enumerate}
Then $\mathcal G$ is strongly hyperbolic relative to the family of maximal parabolic groups.

\end{theorem}

We will be using  following shorter version of this theorem:
\begin{theorem} \label{mjr}
 Let $Q < \out$ be infinite cyclic or a free group of rank 2.
  Suppose  $\{K_s\}$ is a mutually malnormal collection of quasiconvex subgroups of $\F$ such that the conjugacy class
  of each $K_s$ is invariant under $Q$. If the induced tree of coned-off spaces defined by this collection satisfies the hallways flare condition and the cone-bounded hallways strictly flare condition then the extension group $\F \rtimes Q$ is strongly  hyperbolic relative to the  the collection of subgroups $\{K_s\rtimes \widehat{Q}_s\}$, where $\widehat{Q}_s$ is a lift that preserves $K_s$.
\end{theorem}

\begin{proof} Since Theorem \ref{mj-reeves} has more conditions to be satisfied, we will show that in our setting  it is sufficient to have only two of the conditions; the hallways flare
  and the cone-bounded hallways strictly flare conditions to be satisfied to have strong relative hyperbolicity.
  We have the  short exact sequence
  \[1\to\F\to E_{Q} \to Q\to1\]
 where the quotient group is either a free group or an infinite
 cyclic group. The Cayley graph of the quotient group is thus a tree, which enables us to view the Cayley graph $\Gamma$  of  $E_{Q}$ as a tree, $T$, of metric spaces.
 The vertex groups and edge groups of this tree of metric spaces are identified with cosets of $\F$ in $\Gamma$. The maps between the edge space and the two vertex spaces
 (which are the initial and terminal vertices of the edge in consideration) are in fact quasi-isometries in this case. Since each $K_s$ is preserved up to conjugacy, it follows immediately that the \emph{q.i. embedded condition, strictly type preserving condition} and the \emph{q.i. preserving
 electrocution condition} are all satisfied (see \cite[Definition 3.1, 3.2]{MjR-08}).
 
 By the condition imposed on the collection of subgroups $\{K_s\}$ we know that each vertex space and each edge space is (strongly) hyperbolic relative to this collection. Hence $\Gamma$ can be viewed as a tree, $T$, of strongly relatively hyperbolic spaces, denoted again by $\Gamma$.  The collection of maximal parabolic subgroups in this case corresponds to the collection of subgroups $\{K_s\rtimes \widehat{Q}_s\}$.
 Hence if the hallways flare condition and the cone-bounded strictly flares condition is satisfied for $\widehat{\Gamma}$, the conclusion follows.
 \end{proof}

From now on unless we state otherwise, relative hyperbolicity will mean strong relative hyperbolicity. 

 \section{Admissible subgroup systems, Flaring under weak attraction  and relative hyperbolicity in free-by-cyclic groups  }
 \label{SRH}

  \subsection{Standing assumptions, notations and definitions}\label{sec:assump}

Now we will list some definitions, notations and assumptions under which we will prove the theorems in this section.

Let $\phi\in\out$ be rotationless, exponentially growing  and $f: G \to G$ be a CT map representing $\phi$.

\noindent {\bf Critical constant:}\label{critical}

Bounded cancellation notion we use is adapted from  Bestvina-Feighn-Handel's bounded cancellation lemma for train-track representatives of automorphisms of $\F$ ( \cite{BFH-97}); which is inspired by Cooper's same named lemma (\cite{Co-87}).
\begin{definition}
Let  $H_r$ be an exponentially growing stratum with associated Perron-Frobenius eigenvalue $\lambda_r$ and let $BCC(f)$ denote the bounded cancellation constant for $f$. Then, the number

  \[\frac{2BCC(f)}{\lambda_r-1}\]

  is called  the \emph{critical constant} for $H_r$.

\end{definition}
  It is easy to see  that for every number $C>0$ that exceeds the
  critical constant, there is some $1\geq\mu>0$ such that if $\alpha\beta\gamma$ is a concatenation of $r-$legal paths where
   $\beta $ is some $r-$legal segment of length $\geq C$, then the $r-$legal leaf segment of
   $f^k_\#(\alpha\beta\gamma)$ corresponding to $\beta$ has length  at least $\mu\lambda^k{|\beta|}_{H_r}$  (see \cite[pp 219]{BFH-97}).
  To summarize, if we have a path in $G$ which has some $r-$legal ``central''  subsegment of length greater than the
    critical constant, then this segment is protected by the bounded cancellation lemma and under iteration, the
    length of this segment grows exponentially.
    
We will denote by	${|\alpha|}_{H_r}$ the length of an edge-path $\alpha \subset G$ with respect to $H_r$ which is calculated by counting the edges of $\alpha$ in $H_r$ only.

For the rest of this paper, let $C$ be a number larger than the maximum of all critical constants corresponding to EG strata of $f\co G \rightarrow G$.

\noindent {\bf Admissible subgroup system:} \label{K}

Given   $\phi\in \out$, 
let $\mathcal{K} = \{ [K_1], [K_2], \ldots , [K_p]\}$  be a subgroup system and $\mathcal{L}^{+}_{\mathcal{K}}(\phi)$ (respectively, $\mathcal{L}^{-}_{\mathcal{K}}(\phi)$) denote the collection of attracting (respectively, repelling) laminations of $\phi$ whose generic leaves are not carried by $\mathcal{K}$. Assume,
	\begin{enumerate} \label{SA}
 \item $\mathcal{K}$ is a malnormal subgroup system, 
 \item $\mathcal{L}^{+}_{\mathcal{K}}(\phi), \mathcal{L}^{-}_{\mathcal{K}}(\phi)$ are both nonempty, 
		
		\item $\phi([K_s]) = [K_s]$ for each $s\in \{1, \cdots p\}$, 
		\item Let $V^+$ denote the union of attracting neighborhoods of elements of $\mathcal{L}^+_{\mathcal{K}}(\phi)$ defined by generic leaf segments of length $\geq 2C$. Define $V^-$ similarly for $\mathcal{L}^-_{\mathcal{K}}(\phi)$.  By increasing $C$ if necessary,
          \[V^+ \cap V^- = \emptyset\]
		
		\item Every conjugacy class which is not carried by $\mathcal{K}$ is weakly attracted to some element of $\mathcal{L}^+_{\mathcal{K}}(\phi)$.
	
	\end{enumerate}

We will call the subgroups system satisfying the properties above an \emph{admissible subgroup system for $\phi$} (where $\phi\in \out$ is rotationless and exponentially growing). If we are given some finitely generated group $Q = \langle \phi_1, \phi_2, \ldots, \phi_k \rangle$ we say that $\mathcal{K}$ is an \emph{admissible subgroup system for $Q$} if $\mathcal{K}$ is an admissible subgroup system for each $\phi_i$. We will simply write ``admissible subgroup system" when the context is clear.

\noindent {\bf Generalized Weak Attraction:}

Let $\mathcal{K}$ be an admissible subgroup system for some exponentially growing 
 $\phi\in\out$. Since $\mathcal{K}$ is a malnormal subgroup system, by \cite[Lemma 1.11]{HM-20} if every weak limit of every subsequence of $\{\alpha_i\}$ is carried by $\mathcal{K}$, then $\alpha_i$ is carried by $\mathcal{K}$ for all sufficiently large $i$. We use this property to prove an upgraded version of the weak attraction theorem (Lemma \ref{WAT}) which shows that our assumptions on an admissible subgroup system $\mathcal{K}$ are very natural. This theorem also gives us the insight on how to construct such admissible systems (see section \ref{admissconst}).

\begin{theorem}[Generalized weak attraction]\label{modwat}
Let $\phi\in\out$ and let $\mathcal K$ be an admissible subgroup system. Then there exists some $M > 0$ such that for every line or conjugacy class  $\ell$, we have at least one of the following:
\begin{enumerate}
	\item $\ell\in V^-$.
	\item $\ell$ is carried by $\mathcal{K}$.
	\item $\phi^m_\#(\ell) \in V^+$ for every $m\geq M$.
\end{enumerate}
Moreover, a conjugacy class $\alpha$ is weakly attracted to some element of $\mathcal{L}_\mathcal{K}^+(\phi)$ or $\mathcal{L}_\mathcal{K}^-(\phi)$ if and only if it is not carried by $\mathcal{K}$.
\end{theorem}

\begin{proof}
Let $\alpha$ be any conjugacy class in $\F$ that is weakly attracted to some element of $\mathcal{L}^+_\mathcal{K}(\phi)$. Then by item (4) of the condition of being an admissible system,  we know that there is at least some element of $\mathcal{L}^+_\mathcal{K}(\phi)$ whose nonattracting subgroup system does not carry $\alpha$ (see item (2) of Lemma \ref{NAS}). In other words, if a conjugacy class is carried by the nonattracting subgroup system of every element of $\mathcal{L}^+_\mathcal{K}(\phi)$, then that conjugacy class is not weakly attracted to any element of $\mathcal{L}^+_\mathcal{K}(\phi)$ and therefore, necessarily carried by $\mathcal{K}$. The converse also holds. Let $\alpha$ be weakly attracted to some element of  $\mathcal{L}_\mathcal{K}^+(\phi)$. Assume on the contrary that $\mathcal{K}$ carries $\alpha$. Then $\phi^m_\#(\alpha)$ is carried by $\mathcal{K}$ for all $m$. Since the set of lines carried by $\mathcal{K}$ is closed in the weak topology (\cite[Fact 1.8]{HM-20}), every weak limit of every subsequence of $\{\phi^m(\alpha)\}$ must be carried by  $\mathcal{K}$. So $\alpha$ cannot be weakly attracted to any element of $\mathcal{L}^+_\mathcal{K}(\phi)$, otherwise $\mathcal{K}$ would carry an element of $\mathcal{L}^+_\mathcal{K}(\phi)$, contradicting that $\mathcal{K}$ is an admissible system. Hence $\alpha$ cannot be carried by $\mathcal{K}$.  This completes the proof of ``moreover'' part.

Apply Lemma \ref{WAT} repeatedly for each dual lamination pair in $\mathcal{L}^+_\mathcal{K}(\phi)$ and $\mathcal{L}^-_\mathcal{K}(\phi)$, together with the open sets used in constructing $V^+, V^-$,  and let $M$ be greater than maximum of all the exponents obtained as output from the ``moreover''  part of Lemma \ref{WAT}. Now suppose both (1) and (3) fail with this choice of $M$. Then we can conclude that $\ell$ must be carried by the nonattracting subgroup system of every element of $\mathcal{L}^+_\mathcal{K}(\phi)$. Since the set of lines and circuits carried by each of the nonattracting subgroup system is a closed set and $\ell$ is in the intersection of these sets, there is a sequence of conjugacy classes $\{\alpha_i\}$ so that each $\alpha_i$ is carried by the nonattracting subgroup system of every element of  $\mathcal{L}^+_\mathcal{K}(\phi)$ and $\alpha_i$ weakly converges to $\ell$. Arguing as before in the previous paragraph, this means that each $\alpha_i$ is carried by $\mathcal{K}$. Hence, every weak limit of every subsequence of $\{\alpha_i\}$ must be carried by $\mathcal{K}$. Therefore $\ell$ must be carried by $\mathcal{K}$. This completes the proof of the proposition.

\end{proof}
\begin{remark}
Note that our proof above tells us that the set of lines and circuits in $V^+$ and $V^-$ are disjoint from the set of lines and circuits carried by $\mathcal{K}$. Moreover, when $\phi$ is hyperbolic, one may take $\mathcal{K} = \emptyset$ and only option (2) is eliminated from the theorem.

\end{remark}

\noindent {\bf Stallings subgraph of $\mathcal{K}$:}

The \emph{Stallings subgraph of $\mathcal{K}$} with respect to $G$, denoted by $G_\mathcal{K}$, is an immersion $\iota: G_\mathcal{K} \to G$ of a finite graph $G_\mathcal{K}$ whose components are core graph $A_1 \cup \ldots \cup A_p$ such that the subgroup $K_i < \F$ is conjugate to the image $\iota_{*} : \pi_1(A_i) \to \pi_1(G) \cong \F$.  We note that Stallings graphs are unique up to homeomorphism of domains.

\noindent {\bf Legality ratio of paths:} \label{sec:legrate}

The notion of legality ratio was first introduced in \cite[pp-236]{BFH-97} for fully irreducible hyperbolic elements. In the fully-irreducible setting there is only one stratum, and it is exponentially growing. So the notion is a lot simpler. We generalize their definition to  work for all exponentially growing outer automorphisms. The idea is to extend their definition to exponentially growing strata and applying it to a splitting of a given edge path according to their CT structure.

To assist the reader through the technical definition, we present the core idea as follows: given a finite path $\beta$ in $G$, we want to track the ratio the of sum of the lengths of maximal subpaths of $\beta$ which are sufficiently long generic leaf segments of laminations in $\mathcal{L}^+_\mathcal{K}(\phi)$, to the sum of the lengths of subpaths of $\beta$ which lift to the Stallings graph (corresponding to some maximal decomposition). To understand the motivation behind this strategy we refer the reader to Lemma \ref{comparison}.

First we describe how to count the length of a finite path or circuit which does not lift to $G_\mathcal{K}$. 
Let $\beta$ be any finite edge-path in $G$. Consider a decomposition $ \beta_1\beta_2\ldots \beta_p$, of $\beta$ so that for all $1\leq i < p$ if $\beta_i$ lifts to $G_\mathcal{K}$ then $\beta_{i+1}$ does not lift to $G_\mathcal{K}$, and vice versa. 

For a decomposition $\beta^d$ of $\beta$, let $L^d_{\mathcal{K}}(\beta)$ denote the sum of the lengths $\mid \beta_i \mid$ where $\beta_i$ is a component of $\beta^d$ which does not lift to $G_\mathcal{K}$. In other words, 
\[ L^d_{\mathcal{K}}(\beta)= \sum\limits_{\beta_i \in \beta^d_{\mathcal K}} \mid \beta_i \mid  \]
where $\beta^d_{\mathcal K}$ denotes the collection of components of $\beta^d$ that does not lift to  $G_\mathcal{K}$ .


Next we describe how to count legality in a path crossing an exponentially growing stratum.  
To this end, let $\alpha$ be an edge-path in $G$ which does not lift to $G_\mathcal{K}$. Decompose $\alpha$ as a concatenation  $\alpha_1\alpha_2\ldots \alpha_q$ of paths so that for each $1\leq j \leq  q $, either 
\begin{itemize}
    \item $\alpha_j$ does not cross an exponentially growing strata associated to an element of $\mathcal{L}_\mathcal{K}^+(\phi)$, \\
    
or,\\
    
    \item $\alpha_j$ crosses at least one such exponentially growing (EG) strata and both endpoints of $\alpha_j$ are in the same EG strata associated to some element of $\mathcal{L}^+_\mathcal{K}(\phi)$. 
\end{itemize}
Further, we will only consider components $\alpha_j$ of $\alpha$ (if they exist) in this decomposition such that
\begin{enumerate}
    \item $\alpha_j$ is a generic leaf segment, for some element of $\mathcal{L}_\mathcal{K}^+(\phi)$. 
    \item If $\alpha_j$ has height $r$, then ${|\alpha_j|}_{H_r} \geq 2C$ where $ {|\alpha_j|}_{H_r}$ is the $H_r$--length of  $|\alpha_j|$ and C is the constant described earlier in Section \ref{critical}. 
\end{enumerate}
There is a decomposition of $\alpha$ which maximizes the sum $\sum\limits_{j} {\mid \alpha_j\mid}_{H_r}$ where $\alpha_j$ satisfies (1) and (2) above. We denote the maximal sum by $L^*_{leg}(\alpha)$. If $\alpha$ does not have any subpath which is a sufficiently long generic leaf segment of some element of $\mathcal{L}_\mathcal{K}^+(\phi)$, then we define $L^*_{leg}(\alpha) = 0$.

Now we choose a decomposition  $ \beta_1\beta_2\ldots \beta_p$ of $\beta$  and a decomposition $\alpha^i_1\alpha^i_2\ldots \alpha^i_q$ of each $\beta_i \in \beta^d_{\mathcal K}$ as above so that the  ratio $\frac{\sum\limits_i L^*_{leg}(\beta_i)}{L^d_\mathcal{K}(\beta)}$ is maximized. Let
\[ LEG_\mathcal{K}(\beta):= \displaystyle\frac{ L_{leg}(\beta)}{L_\mathcal{K}(\beta)} :=\max \big{\{}\frac{\sum\limits_i L^*_{leg}(\beta_i)}{L^d_\mathcal{K}(\beta)}\big{\}}\]
where the numerator $L_{leg}(\beta)$ computes legal part of $\beta$ relative to $\mathcal{K}$ and the denominator $L_\mathcal{K}(\beta)$ computes the total length of $\beta$ relative to $\mathcal{K}$.

  Note that the components which do not cross any exponentially growing stratum associated to elements of $\mathcal{L}^+_\mathcal{K}(\phi)$  and concatenations of fixed edges and indivisible Nielsen paths are  ignored in computation of $L_\mathcal{K}(\beta)$.

In a symmetric fashion working with a CT map $g: G' \to G'$ for $\phi^{-1}$ and using the set of laminations $\mathcal{L}^-_\mathcal{K}(\phi)$, we define the notions of legality ratio for a circuit and corresponding legality and length computations in the backward direction. We use the notations $LEG^-_\mathcal{K}(\phi^{-m}_\#(\beta)), L^-_{leg}(\beta), L^-_\mathcal{K}(\beta)$ to denote the aforementioned quantities. 

\subsection{Free-by-cyclic relatively hyperbolic extensions}

Let $\phi\in\out$ be exponentially growing and satisfy the conditions in Section \ref{sec:assump} (i.e, with an admissible subgroup system $\mathcal{K}$).
In this section we will show that under the assumptions of the previous section, the extension of $\langle \phi \rangle $  is strongly relatively hyperbolic with respect to the finite family of subgroups $\{K_s \rtimes \langle \Phi_s \rangle\}_{s=1}^p$, where $\Phi_s\in \aut$ is some lift of $\phi$ which preserves $K_s\in \mathcal K$, where $\mathcal K $ is an admissible subgroup system for $\phi$.

To this end, we will first establish under which conditions legal paths grow with the iteration of under $\phi$. This subsection borrows from \cite[Section 2]{Gh-23} and we will refer the reader to the relevant sections of that paper for most of the technical details of the proofs. We however will give full details of the proofs of two crucial results regarding the comparison of the lengths (Lemma \ref{comparison}) and the growth of the legality ratio (Lemma \ref{legality}). Once these are established, the proofs of Lemma \ref{flare}, Proposition \ref{conjflare}, Proposition \ref{strictflare} and Proposition \ref{cbhfc} can be very easily derived from the proofs of \cite{Gh-23}.

\subsubsection{\bf{Electrified metric, length bounds and growth of legality under iteration. }}

We electrify $\F$ with respect to the collection $\mathcal{K}$ to obtain a strongly relatively hyperbolic group $\widehat{\F}$. From now on, we shall use the notation ${\mid \cdot \mid}_{el}$ to denote the length of a word of $\F$ in this electrified metric. The notation ${\vert\vert \alpha \vert\vert}_{el}$ will denote the length of the shortest representative of the conjugacy class $\alpha$ in the  electrified  metric space $(\widehat{\F},{\mid \cdot \mid}_{el})$.

Consider the universal cover $\widetilde{G}$ of $G$. Let  $G_\mathcal{K}$ denote the Stallings graph  whose components are the core graph $A_1 \cup \ldots \cup A_p$. For each $j$ we lift the maps $\iota : A_j \to G $ to the universal cover and   electrify all copies of translates of the images of these lifts. Since $\mathcal{K}$ is malnormal, the resulting  electrified  metric space $\big(\widetilde{G}, {|\cdot|}^{\widetilde{G}}_{el}\big)$ is strongly relatively hyperbolic. We also note that there is an $\F$-equivariant quasi-isometry between the  electrified  metric spaces $\big(\widehat{\F},{\mid \cdot \mid}_{el} \big)$  and $\big(\widetilde{G}, {|\cdot|}^{\widetilde{G}}_{el}\big)$. It is this quasi-isometry that makes the next lemma work, where we connect the two different length computations for a conjugacy class $\alpha$: one comes from the  electrified  metric described above and the other, $L_{\mathcal{K}}(\alpha)$,  comes from the definition of legality ratio, where we look at $\alpha$ as a circuit in $G$. The lemma shows that the two notions of length closely track one another and allows us to use the train-track machinery to prove flaring. 
In the following lemma (and later on in the paper) we use $g: G'\to G'$ to denote a CT map for $\phi^{-1}$.
\begin{lemma}[Length comparison]
   \label{comparison}
Let $\phi\in \out$ be as in section \ref{sec:assump} and $\mathcal{K}$ an admissible subgroup system for $\phi$. If $\alpha$ is any circuit in $G$ which is not carried by $\mathcal{K}$ then
 there exists some $J>0$, independent of $\alpha$, such that
 \[ J \geq  L_{\mathcal{K}}(\alpha)/{||\alpha||}_{el}\geq \frac{1}{J}\,\,\text{and}\, \, \, J\geq  L^-_{\mathcal{K}}(\alpha)/{||\alpha||}_{el}\geq \frac{1}{J}.\]
\end{lemma}

\begin{proof}

Recall that $L_\mathcal{K}(\alpha)$ is obtained by adding in the maximal decomposition of $\alpha$, the lengths of subsegments of $\alpha$ which do not lift to $G_\mathcal{K}$. If $\alpha$ is not carried by $\mathcal{K}$,  then $L_\mathcal{K}(\alpha) \geq 1$.

The proof is almost identical to the  proof of \cite[Lemma 3.7]{Gh-23} only difference being how we calculate $L_\mathcal{K}(\alpha)$. In \cite[Lemma 3.7]{Gh-23} $L_\mathcal{K}(\alpha)$ was calculated when $\mathcal{L}^+_\mathcal{K}(\phi)$ had just one element and we are extending it to the case where $\mathcal{L}^+_\mathcal{K}(\phi)$ can have finitely many elements.   The idea is to decompose a lift of a maximal (in the above sense) $\alpha$ in the  electrified  metric space $\widetilde G$ into concatenation of segments that are outside of cosets and that are geodesics of length one connecting two points inside the copy of a
coset through the  conepoint. The number of both types of segments is bounded and this number bounds the  electrified lengths from above.

The other direction requires taking any cyclically reduced word  $\omega$ in  $\mathbb F \setminus \bigcup \mathcal K_i$, and decomposing the path between 1 and $\omega$ in $G$ into segments as above. In this case, as lifts of some segments might be included in $G_{\mathcal K}$ and hence might get shorter after tightening in $G$;   there will  be cancellations between some of those subsequent segments.  This bounds  $L_\mathcal{K}$ by the  electrified  length from below. This gives us $J_1 > 0$, which satisfies the first inequality.

Similarly we can find $J_2 > 0$ such that the second inequality is satisfied. Now take $J = \text{max }\{J_1, J_2\}$ to complete the proof. 

\end{proof}
Following lemma establishes the fact that for certain classes of circuits $\beta$, the number $L_{leg}(\beta)$ will be a significant proportion of $L_{\mathcal{K}}(\beta)$, after perhaps iterating by a uniformly bounded exponent.

\begin{lemma}[Legality growth]
\label{legality}
Let $\phi\in \out$ be as in section \ref{sec:assump} and $\mathcal{K}$ an admissible subgroup system for $\phi$. Then there exists numbers  $\epsilon >0$ and $N_0>0$ such that for every circuit  $\beta$ in $G$ such that $\beta\in V^+ $ and $\beta\notin V^-$,  we have $LEG_\mathcal{K}(f^n_\#(\beta)) \geq \epsilon$ for all $n \geq N_0$.\\
Similarly switching the roles of $V^+$ and $V^-$, we have $LEG^-_\mathcal{K}(g^{n}_\#(\beta)) \geq \epsilon$ for all $n \geq N_0$.
\end{lemma}

\begin{proof} Supposing that the conclusion is false, we give an argument by contradiction. Suppose we have a sequence of integers $n_j \to \infty$ and corresponding sequence of circuits $\alpha_j$ satisfying the hypothesis and such that $LEG_\mathcal{K}(f^{n_j}_\#(\alpha_j)) \to 0$.  Since $\alpha_j \in V^+$ we have that $LEG_\mathcal{K}(\alpha_j) \neq 0$ for every $j$. Therefore we may assume $L_{\mathcal{K}}(\alpha_j) \to \infty$. Since $LEG_\mathcal{K}(f^{n_j}_\#(\alpha_j)) \to 0$,  $\alpha_j$'s do not contain sufficiently many long subpaths which are generic leaf segments of elements of $\mathcal{L}^+_\mathcal{K}(\phi)$ for the legality to grow under iterates of $f_\#$.  Therefore as $j\to \infty$ subpaths of $\alpha_j$ which are not generic leaf segments of $\mathcal{L}^+_\mathcal{K}(\phi)$ become arbitrarily large since $L_\mathcal{K}(\alpha_j)\to \infty$. Now we choose subpaths $\delta_j$ of $\alpha_j$ such that the following hold:
 \begin{enumerate}
     \item $\delta_j \notin V^-$ and $L_\mathcal{K}(\delta_j) \to \infty$.
     \item  $LEG_\mathcal{K}(f^{n_j}_\#(\delta_j)) = 0$.

 \end{enumerate}

 Since $L_\mathcal{K}(\delta_j)\to \infty$ (hence $|\delta_j| \to \infty$)  we may assume that $\delta_j$ is a circuit for all sufficiently large $j$. This tells us $\delta_j$'s are not carried by $\mathcal{K}$ and so each $\delta_j$ is weakly attracted to some element of $\mathcal{L}^+_\mathcal{K}(\phi)$.  Since there are only finitely many elements in $\mathcal{L}^+_\mathcal{K}(\phi)$, applying item (2) of Lemma \ref{WAT} and passing to a subsequence if necessary we may assume that $\delta_j$'s are not carried by the non-attracting subgroup system corresponding to some fixed attracting lamination  $\Lambda^+\in\mathcal{L}^+_\mathcal{K}(\phi)$.

 Then by the Lemma \ref{NAS} (5) there exists a weak limit $\ell$ of the $\delta_j$'s  such that  $\ell$ is not carried by $\mathcal{A}_{na}(\Lambda^+)$. Since $\delta_j \notin V^-$ we have $\ell \notin V^-$, as $V^-$ is an open set. This implies that $\ell$ is not in the attracting neighborhood of the dual lamination $\Lambda^-$ of $\Lambda^+$, which is contained in $V^-$.   Lemma \ref{WAT} applied to the dual lamination pair  $\Lambda^+, \Lambda^-$   then implies that $f^{n_j}_\#(\ell) \in V^+$ for all $j$ sufficiently large. Since $V^+$ is an open set, there exists some $J>0$ such that $f^{n_j}_\#(\delta_j)\in V^+$ for all $j\geq J$.  This  violates item (2) above.

 A similar argument gives us a threshold exponent for $\phi^{-1}$ and we let $N_0$ be the maximum of these two thresholds.
\end{proof}

\begin{proposition}[Exponent control]
\label{flare}
Let $\phi\in \out$ and $\mathcal{K}$  satisfy the hypothesis of Lemma \ref{legality}.Then  for every $ A>0$, there exists $N_1 > 0$ such that for every circuit  $\beta$ in $G$ with the property that  $\beta\in V^+$, $\beta\notin V^-$ we  have $L_{\mathcal{K}}(f^n_\#(\beta)) \geq A . L_{\mathcal{K}}(\beta)$ for all $n \geq N_1$.

Similarly, switching the roles of $V^+$ and $V^-$, we have $L^-_{\mathcal{K}}(g^n_\#(\beta)) \geq A . L^-_{\mathcal{K}}(\beta)$ for all $n \geq N_1$.
\end{proposition}

\begin{proof}
 By Proposition \ref{legality}, there exists  $N_0$ such that for any circuit $\beta$ satisfying the hypothesis we have $LEG_\mathcal{K}(f^n_\#(\beta)) \geq \epsilon$ for all $n \geq N_0$.  Let $\alpha = f^{N_0}_\# (\beta)$. By taking a decomposition of  $\alpha$ as in the
 definition of legality, we obtain $L_{leg}(\alpha) \geq \epsilon L_{\mathcal{K}}(\alpha)$. If $\lambda$ is the minimum of the stretch factors corresponding to the exponentially growing strata of $f$, we get $$L_{\mathcal{K}}(f^t_\#(\alpha)) \geq L_{leg}( f^t_\#(\alpha)) \geq D \lambda^t L_{leg}(\alpha) \geq D\lambda^t \epsilon . L_{\mathcal{K}}(\alpha)$$ for some constant $0< D\leq 1$ arising out of  the critical constant (see the role of $\mu$ in discussion after Definition \ref{critical}).  Since $N_0$ is fixed, we may choose $N'_1$ large enough, independent of $\beta$ (due to the bounded cancellation property),  such that $D\lambda^{N'_1} \epsilon L_\mathcal{K}(\alpha) \geq A L_{\mathcal{K}}(\beta)$. The result then follows for all $n \geq N'_1$.

 By a symmetric argument, we obtain $N_1''$ satisfying the inequality for $L^-_{\mathcal{K}}(g^n_\#(\beta)) \geq A . L^-_{\mathcal{K}}(\beta)$ for all $n\geq N_1''$. Now take $N_1$ to be the maximum of $N_1', N_1''$ to finish the proof. 
 \end{proof}


\subsubsection{\bf{Flaring conditions.}}
In this section we will extend the flaring conditions known  for fully irreducible automorphisms (conjugacy flaring and hallways flaring) to those of exponentially growing ones relative to an admissible subgroup system. These results will lead us to a form of bounded coset penetration property needed in relative hyperbolicity called \emph{cone bounded hallways flare condition}.

\begin{proposition}[Conjugacy flaring]
\label{conjflare} Let $\phi\in \out$ be as in section \ref{sec:assump} and $\mathcal{K}$ an admissible subgroup system for $\phi$. 
 There exists some $N_2>0$ such that for every conjugacy class $\alpha$ not carried by $\mathcal{K}$, we have
  $$3{||\alpha||}_{el} \leq \mathsf{max} \{{||\phi^m_\#(\alpha)||}_{el}, {||\phi^{-m}_\#(\alpha)||}_{el}\} $$
 for every $m\geq N_2$.
\end{proposition}

\begin{proof}
The idea of the proof is similar to that of \cite[Proposition 3.7]{Gh-23}, we will  summarize the argument here highlighting the differences.

We observe that since $\alpha$ is not carried by $\mathcal{K}$, by Theorem \ref{modwat}  either $\alpha\in V^-$ or there exists $M>0$ such that $\phi^M_\#(\alpha)\in V^+$.  Since $V^+, V^-$ are disjoint sets, we have either $\alpha\in V^- $ and $\alpha \notin V^+$ or $\phi^M_\#(\alpha)\in V^+$ and $ \phi^{-M}_\#(\alpha)\notin V^-$.
 Let $N'_0,  \epsilon$ be as in output of Lemma \ref{legality}.  We have either ${LEG}_\mathcal{K}(\phi^{M+N'_0}_\#(\alpha)) \geq \epsilon$ or ${LEG}^-_\mathcal{K}(\phi^{-N'_0}_\#(\alpha)) \geq \epsilon$. Let $N_{0}:=M+N'_0$. Then for all $m\geq N_0 $ have either ${LEG}_\mathcal{K}(\phi^{m}_\#(\alpha)) \geq \epsilon$ or ${LEG}^-_\mathcal{K}(\phi^{-m}_\#(\alpha)) \geq \epsilon$.

 Let $\lambda_+$ (respectively $\lambda_-$) be maximum of stretch factors associated to the EG strata for elements of $\mathcal{L}^+_\mathcal{K}(\phi)$ (respectively $\mathcal{L}^-_\mathcal{K}(\phi)$). Every conjugacy class that is not carried by $\mathcal{K}$ is stretched at most by a factor of  $\lambda_+$ under $\phi$ and by a factor of $\lambda_-$ under $\phi^{-1}$. So for every circuit $\alpha$ not carried by $\mathcal{K}$, we get
  $L_\mathcal{K}(\phi^{N_0}_\#(\alpha)) \leq \lambda^{N_0}_+ L_\mathcal{K}(\alpha)$, which implies $L_\mathcal{K}(\alpha) \leq \lambda_+^{N_0} L_\mathcal{K}(\phi^{-N_0}_\#(\alpha))$ by replacing $\alpha$ with $\phi^{-N_0}_\#(\alpha)$.  By a symmetric argument, switching the roles of $\phi$ and $\phi^{-1}$, we will get an inequality involving $\lambda_-$. Use Lemma \ref{comparison} to choose some number $J_1>0$ such that we have
  ${||\phi^{N_0}_\#(\alpha)||}_{el} \geq {||\alpha||}_{el}/J_1$ and
  ${||\phi^{-N_0}_\#(\alpha)||}_{el} \geq {||\alpha||}_{el}/J_1$ for every conjugacy class $\alpha$ not carried by $\mathcal{K}$.  Note that $J_1$ is chosen arbitrarily here, only to track the ratio of the numbers above and it depends only on $N_0, \lambda_+, \lambda_-$.
  
  Now we repeat the argument in \cite[Proposition 3.7]{Gh-23} to complete the proof as follows.
  We apply Lemma \ref{comparison} again to compare  $L_\mathcal{K}(\phi^{N_0}_\#(\alpha)) $ with ${||\phi^{N_0}_\#(\alpha)||}_{el}$:
  for some constant $J_2$ and every conjugacy class $\alpha$
  as above,  either  $J_2 \geq L_\mathcal{K}(\phi^{N_0}_\#(\alpha))/{||\phi^{N_0}_\#(\alpha)||}_{el}\geq \frac{1}{J_2} \text{ or }
  J_2\geq L^-_\mathcal{K}(\phi^{-N_0}_\#(\alpha))/{||\phi^{-N_0}_\#(\alpha)||}_{el} \geq \frac{1}{J_2}$.
Without loss of generality assume that ${LEG}_\mathcal{K}(\phi^m_\#(\alpha)) \geq \epsilon$. Then by applying
  Lemma \ref{flare} with $\epsilon$, $A= 3J_1J_2^2$ and $N_2=N_0+N_1$ for all
  $m\geq N_2$, we have
\begin{equation}
  \begin{split}
  {||\phi^m_\#(\alpha)||}_{el} & \geq  3{||\alpha||}_{el}
   \end{split}
\end{equation}


\end{proof}

\begin{proposition}[Hallways flaring]
 \label{strictflare}
Let $\phi\in \out$ be as in section \ref{sec:assump} and $\mathcal{K}$ an admissible subgroup system for $\phi$. 
  Choose a lift $\Phi\in\aut$ of $\phi$.
  Then there exist numbers $N_\Phi >0$ and $L_\Phi >0$ such that for every word $w\in \F \setminus \bigcup\limits_i K_i$
  with ${|w|}_{el} \geq L_\Phi$  we have
  $$2{|w|}_{el} \leq \mathsf{max}\{{|\Phi^n_\#(w)|}_{el}, {|\Phi^{-n}_\#(w)|}_{el}\}$$
  for every $n \geq N_\Phi$.
  Moreover, if $w$ is cyclically reduced then the result holds for all $w$ such that ${|w|}_{el} > 1$.
 \end{proposition}

\begin{proof}  There are finitely many attracting laminations in $\mathcal{L}^+_\mathcal{K}(\phi)$, none of which are carried by the the subgroup system $\mathcal{K}$. Since $\phi$ leaves the free factor support of the attracting laminations in $\mathcal{L}^+_\mathcal{K}(\phi)$ invariant, the lift $\Phi$ will leave some free factor representing the corresponding free factor support invariant. For cyclically reduced words, the conjugacy flaring (Proposition \ref{conjflare}) already implies the conclusion of this proposition. \\
To deal with words which are not cyclically reduced, let $B$ denote the union of basis of these invariant free factors such that if $b\in B$ then $[b]$ is not carried by $\mathcal{K}$. If $w$ is not cyclically reduced then choose some cyclically reduced $b$ such that $bw\in \F\setminus\cup K_t$. 
Now we  repeat the argument of \cite[Proposition 3.10]{Gh-23} on the word $bw$ to get the conclusion.

\end{proof}

\begin{lemma}\label{conjugatorstretch}
Let $\phi\in \out$ be as in section \ref{sec:assump} and  $\mathcal{K} = \{ [K_1], [K_2], \ldots , [K_p]\}$   an admissible subgroup system for $\phi$. 
If we have a lift $\Phi$ such that $\Phi(K_s)= K_s$  and $\Phi(K_t) = x_t^{-1} K_t x_t$, for some
 $x_t\in \F$ and $s\neq t$, then
 \begin{enumerate}
 	\item $x_t \notin K_s \cup K_t$.
 	 \item ${|\Phi^n_\#(x_t)|}_{el}$ must grow exponentially fast as $n\to \infty$. Moreover $[x_t]$ is not carried by $\mathcal{K}$. 
	\item The  electrified  length ${|x_t \Phi(x_t) \Phi^2(x_t)\ldots \Phi^{n-1}(x_t)|}_{el} \to \infty$ exponentially fast as $n\to \infty$.
\end{enumerate}

\end{lemma}
\begin{proof}
First we prove (1). Observe that $x_t \notin K_t$. For otherwise $\Phi$ leaves both $K_s$ and $K_t$ invariant and this will violate item (5) of our list of conditions of being admissible  \ref{K} since $[K_s * K_t]$ will be left invariant by $\phi$. Similarly note that $x_t \notin K_s$, for otherwise there will be a lift of $\phi$ fixing both $K_s$ and $K_t$.  

To prove (2), let $a_s \in K_s, a_t\in K_t$ be nontrivial elements. Observe that the conjugacy class of  $w = a_s x_t^{-1} a_t x_j$ is not carried by $\mathcal{K}$ and so $[\Phi^n_\#(w)]$  weakly converges to some element of $\mathcal{L}^+_\mathcal{K}(\phi)$ as $n\to\infty$. Therefore  $L_\mathcal{K}([\Phi^n_\#(w)])$ grows exponentially  fast as $n\to \infty$. Since $K_s$ and $K_t$ are invariant under $\phi$, the only way this can happen is if $L_\mathcal{K}([\Phi^n_\#(x_t)])$  grows exponentially fast as $n$ increases. By using Lemma \ref{comparison}, we get the desired conclusions.

The third conclusion follows directly from the proof of the second conclusion by computing iterates of $w$ under $\Phi$.

\end{proof}

Now we will work on the cone locus of $\widehat\Gamma$; the induced tree of coned off spaces whose vertices and the edges here are the cone points in copies of
cosets of $\widehat \F$ (Section \ref{sec:relhyp}). For a lift $\Phi_i$ of $\phi$ that leaves $K_i$ invariant, the maximal cone subtree  correspond to cosets of  $\F \rtimes \langle \Phi \rangle$. 

\begin{remark}
Observe that  in the mapping torus $\F \rtimes \langle \Phi \rangle$, we have 
\[x_t \Phi(x_t) \Phi^2(x_t)\ldots \Phi^{n-1}(x_t) = \Phi_t^n \Phi^{-n}\]
where $\Phi = x_t^{-1} \Phi_t$ as a group element of $\F \rtimes \langle \Phi \rangle$. As elements of $\aut$  one may think of $\Phi_t$ as the composition of maps $\iota_{x_t^{-1}} \circ\Phi$. Here $\iota_{x_t^{-1}} $ is the inner automorphism $w \mapsto x_t w x_t^{-1}$.

\end{remark}

 \begin{proposition}[Cone-bounded hallways flaring] \label{cbhfc}
   Suppose $\phi\in \out$ be as in section \ref{sec:assump} and  $\mathcal{K} = \{ [K_1], [K_2], \ldots , [K_p]\}$  an admissible subgroup system for $\phi$. 
   Let $\Phi_s$ be a lift such that $\Phi_s(K_s)=K_s$ and $\Phi_s(K_t) = x_t^{-1} K_t x_t$ for some $x_t\in \F$ where $t\neq s$. Then there exists some $N_c > 0$ such that

    \[ 2{|x_t|}_{el} \leq \mathsf{max}\{{|{\Phi_s}^n_\#(x_t)|}_{el}, {|{\Phi_s}^{-n}_\#(x_t)|}_{el}\}   \]
   for every $n\geq N_c$ (independent of $s,t$) and for every such $t\neq s$.

\end{proposition}

\begin{proof}
By (1) of Lemma \ref{conjugatorstretch} we have  $x_t \in \F \setminus \cup K_q$. We also note that $x_t$ has exponential growth under iterates of   $\Phi_s$ in the  electrified metric in $\widehat{\F}$ for each $s\neq t$. We apply  Proposition \ref{strictflare} to each $\Phi_s$.
Let $L' = \text{ max } \{L_s\}_{s=1}^p$ and $N' = \text{ max }\{N_s\}_{s=1}^p$, where $L_s$ (word-length threshold) and $N_s$ (exponent threshold).

If $y_t$ is a word in $\F$ such that
  $\Phi_s(K_t) = y_t^{-1} K_t y_t$ then we have $x_t^{-1}K_tx_t = y_t^{-1}K_t y_t$. Since $\mathcal{K}$ is malnormal, it follows that $y_t = x_t^{-1}a_t$ for some $a_t\in K_t$. This means that once $\Phi_s$ is chosen, the corresponding conjugators are uniquely determined up to choosing the tail $a_t$. Without loss of generality suppose that $x_t^{-1}$ has no tail in $K_t$. If $y_t = x_t^{-1}a_t$ for some $a_t\in K_t$, then $|y_t|_{el} = |x_t|_{el} +1$ and $\Phi^n_s(y_t) = \Phi^n_s(x_t^{-1}) \ldots \Phi_s(x_t^{-1}) x_t^{-1}  a_{n, t} x_t \ldots \Phi^{n-1}_s(x_t)$ for some $a_{n, t}\in K_t$. Since $\mathcal{K}$ is malnormal, for no $n > 0$ can we have $\Phi_s(a_{n, t})= x_t a_{n, t} x_t^{-1}$.
  Therefore we have shown that ${|\Phi^n_s(y_t)|}_{el}$ grows exponentially if and only if ${|x_t \Phi_s(x_t) \Phi^2_s(x_t)\ldots \Phi^{n-1}_s(x_t)|}_{el}$ grows exponentially. Lemma \ref{conjugatorstretch} now tells us that this is always true under our hypothesis.

If ${|x_t|}_{el} > L'$, then Proposition \ref{strictflare} does the job. If ${|x_t|}_{el} \leq L'$, then using Lemma \ref{conjugatorstretch}, we get constants $N'_{s, t}$ such that ${|\Phi_s^{N'_{s, t}}(x_t)}|_{el} > 2L' > {|x_t|}_{el}$ and ${|x_t \Phi_s(x_t) \Phi^2_s(x_t)\ldots \Phi^{N'_{s, t}}_s(x_t)|}_{el} > 2L' > {|x_t|}_{el}$. Now we can apply Proposition \ref{strictflare}  by replacing the $x_t$'s with an iterate if necessary. Let $N_s$ be the maximum of the integers $N'$ and $N'_{s, t}$ as $t$ varies. For the final step let $N_c$ be the max of $N_s$'s, as $s$ varies. This completes the proof.

\end{proof}

Finally, we are ready to prove the relative hyperbolicity of free-by-cyclic groups under the standing assumptions.

\begin{fbcrh}
Let $\phi\in\out$ and  $\mathcal{K} = \{ [K_1], [K_2], \ldots , [K_p]\}$   be an admissible subgroup system for $\phi$.

Then, the group $\F \rtimes \langle \phi \rangle$ is strongly hyperbolic relative to the collection of subgroups $\{K_s \rtimes \langle \Phi_s \rangle\}_{s=1}^p$, where $\Phi_s\in \aut$ is a lift of $\phi$ such that $\Phi_s (K_s) = K_s$.

\end{fbcrh}

\begin{proof}

 Proposition \ref{strictflare} proves the hallway flaring condition and Proposition \ref{cbhfc} proves the cone-bounded hallway flaring condition required by the Mj-Reeves theorem. Now apply Lemma \ref{mjr} to complete the proof.
 \end{proof}

Observe that if $\Lambda^+_\phi$ is an attracting lamination of $\phi$, then the nonattracting subgroup system $\mathcal{A}_{na}(\Lambda^+_\phi)$ is an admissible subgroup system. We get the following corollary.

 \begin{corollary}\cite[Proposition 3.13]{Gh-23}
 \label{relhypext}
  Let $\phi\in\out$ be exponentially growing outer automorphism equipped with a dual lamination pair $\Lambda^\pm_\phi$, which
  are topmost. Let $\Phi\in\aut$ be a lift of $\phi$.
  Also let $\mathcal{A}_{na}(\Lambda^\pm_\phi) = \{[K_1], [K_2],\ldots, [K_p]\} $ denote the nonattracting subgroup system for $\Lambda^\pm_\phi$.
  If  $K_s$'s denote representatives of $[K_s]$ such that $\Phi_s(K_s)=K_s$ for some lift $\Phi_s$, then the extension group $\F \rtimes \langle \Phi \rangle$
  is strongly hyperbolic relative to the collection of subgroups $\{K_s \rtimes \langle \Phi_s \rangle\}_{s=1}^p$.
 \end{corollary}

 Applying an induction argument to the restriction of $\phi$ on the subgroups $K_s$, one obtains the following result.

 \begin{corollary}\cite[Corollary 3.16]{Gh-23} \label{GL}
  The mapping torus of every  exponentially growing  $\phi\in\out$ is (strongly) hyperbolic relative to a finite collection of peripheral
  subgroups of the form $K_s\rtimes_{\Phi_s} \mathbb{Z}$, where $\Phi_s$ is a lift of $\phi$ that preserves $K_s$
  and the outer automorphism class of $\Phi_s$ restricted to $K_s$ is polynomially growing.
  \end{corollary}

\section {Relatively hyperbolic free-by-free groups from admissible subgroup systems }\label{MtoRH}

In this section we will give sufficient conditions to prove Theorem \ref{main1}. We first modify some definitions from \cite{Gh-23} and \cite{GMj-22} to adapt to our general situation here since we do not use the topmost lamination but \emph{any} lamination.

\begin{definition}\label{relind}
	Let $\phi_1,\cdots   \phi_k $ be a collection of exponentially growing  outer automorphisms of $\F$. Let $\mathcal{K} = \{[K_1], \ldots , [K_p] \} $ be a malnormal subgroup system such that each $[K_s]$ is invariant under the collection $\phi_1,\cdots   \phi_k \in \out$. Let $\mathcal{L}^+_{\mathcal{K}}(\phi_i)$ (resp. $\mathcal{L}^-_{\mathcal{K}}(\phi_i)$) denote the set of attracting laminations of $\phi_i$ (resp. $\phi^{-1}_i$) which are not carried by $\mathcal{K}$ and let  $\mathcal{L}^\pm_\mathcal{K}(\phi_i)$ denote the set $\mathcal{L}^+_\mathcal{K}(\phi_i) \cup \mathcal{L}^-_\mathcal{K}(\phi_i)$.
	
	We say that $\phi_i, \phi_j$, $i\neq j$ are \emph{independent relative to $\mathcal{K}$ } if generic leaves of  elements of $\mathcal{L}^\pm_\mathcal{K}(\phi_i)$ and  $\mathcal{L}^\pm_\mathcal{K}(\phi_j)$ are not asymptotic to each other.
 We will say that a collection $\phi_1,\cdots, \phi_k \in \out$ is independent relative to  $\mathcal K$ if every pair $\phi_i, \phi_j$, $i\neq j$ in the collection is independent relative to $\mathcal K$. 

\end{definition}

 \begin{lemma}[Disjoint neighborhoods exist]\label{disjointnbds}
Let $\phi_1,\cdots , \phi_k$  be a collection that is independent relative to an admissible subgroup system $\mathcal{K}$. Then,
\begin{enumerate}
	\item For  all $i\in \{1,\cdots k\}$ there exists neighborhoods $V^+_{\phi_i}$  and $V^-_{\phi_i}$, of generic leaves of every element of $\mathcal{L}^+_\mathcal{K}(\phi_i)$
 and $\mathcal{L}^-_\mathcal{K}(\phi_i)$ respectively such that $V^+_{\phi_i} \cap V^-_{\phi_i} = \emptyset$. 
	\item $V^+_{\phi_i}, V^-_{\phi_i}, V^+_{\phi_j}, V^-_{\phi_j}$, $i\neq j$ can be chosen so that they are pairwise disjoint.
	\item $\mathcal{L}^\pm_\mathcal{K}(\phi_i) \cap \mathcal{L}^\pm_\mathcal{K}(\phi_j) = \emptyset$, for each $i\neq j$.
\end{enumerate}
 Moreover, we have $(3) \implies (2) \implies \,\, \phi_1, \ldots, \phi_k$ are independent relative to $\mathcal{K}$. 

\end{lemma}

\begin{proof}

To prove (1) we first choose a sufficiently long generic leaf segment for each  attracting (resp. repelling) lamination in $\mathcal{L}_\mathcal{K}(\phi)$ (resp. in $\mathcal{L}_\mathcal{K}(\phi^{-1})$). Then we choose an attracting (resp. repelling) neighborhood for that generic leaf and take union over all such neighborhoods.


Proof of item (2) is similar to the proof of \cite[Lemma 3.9]{GMj-22}. Fix $i,j \in \{1, \cdots, k\}$. By choosing sufficiently long generic leaf segments of laminations $\Lambda_i, \Lambda_j$, we can find attracting neighborhoods $V^+_i, V^+_j$ of $\Lambda_i, \Lambda_j$ respectively, so that  $V^+_i\cap V^+_j=\emptyset$ for otherwise we would have asymptotic generic leaves, which would  violate the independence hypothesis. Proceeding similarly with other elements of $\mathcal{L}^+_\mathcal{K}(\phi_i)$  we can find disjoint attracting neighborhoods for each pair of attracting laminations from sets $\mathcal{L}^+_\mathcal{K}(\phi_i)$ and  $\mathcal{L}^+_\mathcal{K}(\phi_j)$. Taking  unions of the corresponding attracting neighborhoods gives us $V^+_{\phi_i}, V^+_{\phi_j}$ such that $ V^+_{\phi_i}\cap V^+_{\phi_j} = \emptyset$. Now we repeat the argument for all $i, j \in \{1,\cdots, k\}$ to complete the proof of (2). \\
(3) follows directly from (2). The proof of ``moreover" part follows directly from proof of item (2). 

\end{proof}

\subsection{3--outof--4 stretch}
 For the next proposition we perform a partial electrocution by coning-off $\Gamma$ in the
 short exact sequence $1\to\F\to\Gamma\to Q\to 1$  with respect to the
 collection of subgroups $\{K_s\}$ and an electrocution on $\F$ by coning-off the same collection of subgroups.
 The resulting electric metric ${|\cdot|}_{el}$ on $\widehat{\F}$ is the one used in the statements below.

The following theorem originates from Lee Mosher's work on mapping class groups \cite{Mos-97}.
 For the free groups case it was first shown in \cite{BFH-97} for fully irreducible nongeometric elements.

  \begin{proposition}[3-outof-4 stretch]
 \label{34}
  Let $\phi,\psi \in \out$ be exponentially growing and assume that they are  independent relative to an admissible subgroup system $\mathcal{K}$. We have, 
  \begin{enumerate}
   \item There exists some $M > 0$ such that
   for any conjugacy class $\alpha$ not carried by $\mathcal{K}$,  at least three of the four numbers
   $$ {||\phi^{n}_\#(\alpha)||}_{el}, {||\phi^{-n}_\#(\alpha)||}_{el},
  {||\psi^{n}_\#(\alpha)||}_{el}, {||\psi^{-n}_\#(\alpha)||}_{el}$$
  are greater than or equal to $3{||\alpha||}_{el}$, for all $n \geq M$.

  \item Given lifts $\Phi, \Psi$, there exist some $L_1 > 0, N > 0$ such that for any word $w\in\F\setminus\bigcup K_i$ with ${|w|}_{el} \geq L_1$, at least three of the four numbers

   $$ {|\Phi^{n}_\#(w)|}_{el}, {|\Phi^{-n}_\#(w)|}_{el},
  {|\Psi^{n}_\#(w)|}_{el}, {|\Psi^{-n}_\#(w)|}_{el}$$
  are greater than $2{|w|}_{el}$, for all $n \geq N$.

  \end{enumerate}

 \end{proposition}

 \begin{proof}
 To prove (1),  
   suppose there does not exist any such $M$. Then,  there is a  sequence of conjugacy classes
  $\alpha_i$ which is not carried by $\mathcal K$, and $n_i>i$ such that at least two of the four numbers
  ${||\phi^{n_i}_\#(\alpha_i)||}_{el}, {||\phi^{-n_i}_\#(\alpha_i)||}_{el}$,
  ${||\psi^{n_i}_\#(\alpha_i)||}_{el}, {||\psi^{-n_i}_\#(\alpha_i)||}_{el}$ are less than
  $3{||\alpha_i||}_{el}$.  By Proposition \ref{conjflare} (conjugacy flaring) at least one of
  $\{{||\phi^{n_i}_\#(\alpha_i)||}_{el}, {||\phi^{-n_i}_\#(\alpha_i)||}_{el}\}$ is $\geq 3{||\alpha_i||}_{el}$ and at least one of  $\{{||\psi^{n_i}_\#(\alpha_i)||}_{el}$, $ {||\psi^{-n_i}_\#(\alpha_i)||}_{el}\}$ is $\geq 3{||\alpha_i||}_{el}$  for all sufficiently large $i$.  Suppose that
  \begin{equation*}
      {||\phi^{n_i}_\#(\alpha_i)||}_{el} \leq 3{||\alpha_i||}_{el} \,\,\,\text{and}\,\,\,
  {||\psi^{n_i}_\#(\alpha_i)||}_{el} \leq 3{||\alpha_i||}_{el}  
  \end{equation*} for all $n_i$.\hfill 

  By Lemma \ref{disjointnbds} we can choose pairwise disjoint neighborhoods  $V^+_\phi, V^-_\phi, V^+_\psi, V^-_\psi $. Since  $V^-_\phi \cap V^-_\psi = \emptyset$, elements of $\{\alpha_i\}$ cannot be in both $V^-_\phi$ and $V^-_\psi$. After passing to a subsequence if necessary suppose that $\{\alpha_i\} \notin V^-_\psi$. Since each $\alpha_i$ is weakly attracted to some element of $\mathcal{L}^+_\mathcal{K}(\psi)$, we may pass to a further subsequence and assume that all elements of the sequence $\{\alpha_i\}$ are weakly attracted to some $\Lambda^+_\psi \in \mathcal{L}^+_\mathcal{K}(\psi)$.

   Hence by using the uniformity part of the weak attraction theorem, there exists some $M_0$
  such that $\psi^m_\#(\alpha_i)\in V^+_\psi$ for all $m\geq M_0$. Choosing $i$ to be sufficiently large
  we may assume that $n_i \geq M_0$ and so by using Lemma \ref{legality} we have ${LEG}_\mathcal{K}(\psi^{n_i}_\#(\alpha_i))\geq \epsilon$ for some $\epsilon>0$.
  By using Lemma \ref{flare} we obtain that for any $A>0$ there exists some $M_1$ such that
  $$L_\mathcal{K}(\psi^s_\#(\alpha_i)) \geq A. L_\mathcal{K}(\alpha_i) $$ for every $s>M_1$. Therefore,
 $L_\mathcal{K}(\psi^{n_i}_\#(\alpha_i)) \geq A. L_\mathcal{K}(\alpha_i)$  for all sufficiently large $i$. Choosing a
  sequence $A_i\to \infty$ and after passing to a subsequence of $\{n_i\}$ we may assume that
  $L_\mathcal{K}(\psi^{n_i}_\#(\alpha_i)) \geq A_i. L_\mathcal{K}(\alpha_i)$.  But by Lemma \ref{comparison}, there is a $J>0$ such that we get
  \[ A_i \leq L_\mathcal{K}(\psi^{n_i}_\#(\alpha_i)) / L_\mathcal{K}(\alpha_i) \leq J^2 {||\psi^{n_i}_\#(\alpha_i)||}_{el} /   {||\alpha_i||}_{el} \leq 3 J^2.\]
 This is a contradiction since $A_i \to \infty $.

 By the part (1), there exists some $M > 0$ such that for any conjugacy class $\alpha$ not carried by $\mathcal{K}$, at least three of the four numbers ${||\phi^{n}_\#(\alpha_i)||}_{el}, {||\phi^{-n}_\#(\alpha_i)||}_{el}, {||\psi^{n}_\#(\alpha_i)||}_{el}$, and ${||\psi^{-n}_\#(\alpha_i)||}_{el}$  grows exponentially as $i\to \infty$.

If $w$ is a cyclically reduced representative of some conjugacy class $\alpha$ not carried by $\mathcal{K}$, then $w$ is the shortest representative of its conjugacy class. Hence ${|w|}_{el} = {|| [w]||}_{el}$ and  ${||\phi^{n}_\#([w]])||}_{el} = {|\phi^{n}_\#(w)|}_{el}$. Hence, by using part (1), three of the four lengths ${|\phi^{n}_\#(w)|}_{el}$, ${|\phi^{-n}_\#(w)|}_{el}$, ${|\psi^{n}_\#(w)|}_{el}$, ${|\psi^{-n}_\#(w)|}_{el}$ grows exponentially as $i\to \infty$ and we are done. 

For the next part, assume that $w$ is not cyclically reduced. Proceed as in the proof of Proposition \ref{strictflare} to choose some word $b$, so that $bw\in \F\setminus\cup K_t$ is cyclically reduced and apply the previous argument.
\end{proof}

\begin{corollary}[(2k-1)-outof-2k stretch]\label{2k-1st}
    Let $\phi_1, \ldots ,\phi_k $ be a collection of exponentially growing outer automorphisms that is independent relative to an admissible subgroup system $\mathcal{K}$. We have, 
  \begin{enumerate}
   \item There exists some $M > 0$ such that
   for any conjugacy class $\alpha$ not carried by $\mathcal{K}$, at least $2k-1$ of the $2k$  numbers $\{{||\phi^n_{i_\#}(\alpha)||}_{el}, {||\phi^{-n}_{i_\#}(\alpha)||}_{el}\}_{i=1}^k$ are greater than or equal to $3{||\alpha||}_{el}$, for all $n \geq M$.

  \item Given lifts $\Phi_i$ of $\phi_i$, there exist some $L_1 > 0, N > 0$ such that for any word $w\in\F\setminus\bigcup K_i$ with ${|w|}_{el} \geq L_1$, at least $2k-1$ of the  $2k$ numbers $\{{|\Phi^{n}_{i\#}(w)|}_{el}, {|\Phi^{-n}_\#(w)|}_{el}\}_{i=1}^k$ are greater than $2{|w|}_{el}$, for all $n \geq N$.
  \end{enumerate}
\end{corollary}

\begin{proof} 
   Suppose that the statement $(1)$ is false. Then there exists a sequence of circuits $\alpha_j$, with $\alpha_j$ not carried by $\mathcal{K}$ and a sequence of integers $n_j$ such that at least $2$ of the $2k$ numbers violate the conclusion with the exponents $n_j$. After passing to subsequence if necessary, and without loss of generality suppose that  both ${||\phi^{n_j}_{1_\#}(\alpha_j)||}_{el}, {||\phi^{n_j}_{2_\#}(\alpha_j)||}_{el} < 3 {||\alpha_j ||}_{el}$ as $j \to \infty$. This violates the 3-outof-4 stretch of Proposition \ref{34}, item (1). 

   To prove $(2)$ let $w_j$ be a sequence of words not in $\F \setminus\bigcup K_t$ such that ${|w_j|}_{el} \to \infty$ as $j\to \infty$ and suppose as in part $(1)$, ${|\phi^{n_j}_{1_\#}(w_j)|}_{el}, {|\phi^{n_j}_{2_\#}(w_j)|}_{el} < 2 {|w_j |}_{el}$ as $j \to \infty$. Then this violates the hallway flaring conclusion in item (2) of Proposition \ref{34}. 
\end{proof}

\begin{proposition}[Cone-bounded hallways strictly flare]
  \label{fbfcbhfc}
   Let $\phi,\psi \in \out$ be exponentially growing automorphisms that are independent relative to an admissible subgroup system $\mathcal{K} = \{[K_1], [K_2], \ldots [K_p]\}$. Suppose $\Phi_s(K_s) = K_s = \Psi_s(K_s),  \Psi_t(K_t) = K_t = \Phi_t(K_t)$ and  $\Phi_s(K_t) = x_t^{-1} K_t x_t$ for some $x_t\in \F$ when $1 \leq t\neq s \leq p$.

  Then there exists some $N\geq 0$ such that at least three of the following four numbers  
    \[ {|\Phi^{n}_{s_\#}(x_t)|}_{el}, {|\Phi^{-n}_{s_\#}(x_t)|}_{el},
  {|\Psi^{n}_{t_\#}(x_t)|}_{el}, {|\Psi^{-n}_{t_\#}(x_t)|}_{el}   \]
   are greater than or equal to $2{|x_t|}_{el}$, for every $n\geq N$ (independent of $s, t$) and for every $t\neq s$.

  \end{proposition}

 \begin{proof}

 Lemma \ref{conjugatorstretch} tells us that $x_t \notin K_s \cup K_t$ and ${|{\Phi_i^n}_\#(x_j)|}_{el} \to \infty$ as $n \to \infty$. Arguing similarly as in proof of item (2) in Lemma \ref{conjugatorstretch}, we let $w = a_s a_t$ for some nontrivial $a_s \in K_s , a_t \in K_t$.  Then $[w]$ in not carried by $\mathcal{K}$ and hence by property of admissible systems, $[w]$ is weakly attracted to some element in each of $\mathcal{L}_\mathcal{K}^+(\phi), \mathcal{L}_\mathcal{K}^-(\phi), \mathcal{L}_\mathcal{K}^+(\psi), \mathcal{L}_\mathcal{K}^-(\psi)$. Hence all four of ${L}_\mathcal{K}([{\Phi^n_s}_\#(w)]), {L}_\mathcal{K}([{\Phi^{-n}_s}_\#(w)]), {L}^-_\mathcal{K}([{\Psi^n_t}_\#(w)]), {L}^-_\mathcal{K}([{\Psi^{-n}_t}_\#(w)]) \to \infty$ exponentially fast as $n\to \infty$. If ${|{\Phi_s^n}_\#(x_t)|}_{el}$ does not grow exponentially fast as $n\to \infty$, then ${|{\Phi_s^n}_\#(w)|}_{el}$ cannot grow exponentially fast, implying that $L_\mathcal{K}(\phi^n_\#([w])$ does not grow exponentially fast (by Lemma \ref{comparison}), which is in contradiction to the fact that $[w]$ is weakly attracted to some element of $\mathcal{L}_\mathcal{K}^+(\phi)$. Hence ${|{\Phi_s^n}_\#(x_t)|}_{el} \to \infty$ exponentially fast as $n\to \infty$. By a similar argument we get that  ${|{\Phi_s^{-n}}_\#(x_t)|}_{el} \to \infty$ exponentially fast as $n\to \infty$. Our argument also gives us the fact that $[x_t]$ is not carried by $\mathcal{K}$, hence $[x_t]$ is weakly attracted to some element of $\mathcal{L}_\mathcal{K}^+(\psi)$ and some element of $\mathcal{L}_\mathcal{K}^-(\psi)$. This implies  ${|{\Psi_t^n}_\#(x_t)|}_{el},  {|{\Psi_t^{-n}}_\#(x_t)|}_{el} \to \infty$ exponentially fast as $n\to \infty$. If $x_t$ is replaced by a different conjugator $y_t$, then ${|y_t|}_{el} = 1 + {|x_t|}_{el}$ and arguing exactly as we did in proof of Proposition \ref{cbhfc}, the same properties are true for exponential growth of $y_t$ in the  electrified  metric under iteration. 

Now suppose that our conclusion is false. Then, for concreteness assume that there exists some $s, t$ such that  both ${|\Phi^{n}_{s_\#}(x_t)|}_{el}, {|\Psi^{n}_{t_\#}(x_t)|}_{el} < 2{|x_t|}_{el}$ as $n\to \infty$ - this is a contradiction. This completes the proof.  

 \end{proof}

\begin{corollary}
    Let $\phi_1, \ldots \phi_k \in \out$ be a collection of automorphisms satisfying the properties of section \ref{sec:assump} that is pairwise independent relative to an admissible subgroup system  $\mathcal{K} = \{[K_1], [K_2], \ldots [K_p]\}$. Suppose $\Phi_{1s}(K_s) = \Phi_{2s}(K_s) = \ldots =\Phi_{ks}(K_s) = K_s$ and $\Phi_{1s}(K_t) = x_t ^{-1} K_t x_t$ for some $x_t\in \F$ whenever $1 \leq t\neq s \leq p$. Then there exists some $N > 0$ such that  $2k-1$ of the $2k$ numbers  $\{{|\Phi^n_{{is}_\#}(x_t)|}_{el}, {|\Phi^{-n}_{{is}_\#}(x_t)|}_{el}\}_{i=1}^k$ are greater than $2{|x_t|}_{el}$ for every $n > N$ (independent of $s, t$)
\end{corollary}

\begin{proof}
    We  prove this by contradiction. First note that ${|\Phi^{n}_{is_\#}(x_t)|}_{el} \to \infty$  exponentially fast as $n\to \infty$ as seen in proof of Proposition \ref{fbfcbhfc}.  Suppose that the conclusion is false. Without loss of generality suppose ${|\Phi^{n_t}_{1s_\#}(x_t)|}, {|\Phi^{n_t}_{2s_\#}(x_t)|} < 2 {|x_t|}_{el}$ as $n_t\to \infty$. This gives us a contradiction. 
\end{proof}
\subsection{Relatively hyperbolic extensions from admissible subgroup systems}

The lemma below will provide one direction of the Theorem \ref{main1}. 

\begin{lemma}\label{fbfrh}
   Let $\phi_1, \ldots \phi_k \in \out$ be a collection of automorphisms as in section \ref{sec:assump} that are independent  relative to an admissible subgroup system  $\mathcal{K} = \{[K_1], [K_2], \ldots [K_p]\}$. Then,

  \begin{enumerate}
  	\item There exists some $M\geq 1$ such that for all $m_i \geq M$, $Q= \langle \phi_1^{m_1},\cdots,  \phi_k^{m_k} \rangle$ is a free group.  Moreover, every element of the group $Q$ is exponentially growing and satisfies the standing assumptions of section \ref{sec:assump}.
 	\item $\F\rtimes \widehat{Q}$ is strongly relatively hyperbolic relative to the
  collection of subgroups $\{K_s \rtimes \widehat Q_s \}$ $s\in \{1, \cdots, p\}$, where  $\widehat Q_s$  is a lift of $Q_s$ that preserves $K_s$
  and $\widehat{Q}$ is any lift of $Q$.
  	
  	  	\item Let  $\xi = \phi^{m_1}\psi^{n_1}\ldots \psi^{n_s}$ for some $\phi, \psi \in Q$.  When $M$ is sufficiently large, every generic leaf of every element of $\mathcal{L}_\mathcal{K}^+(\xi)$ is in $W^+_\phi$ if $m_1 > 0$ and in $W^-_\phi$ if $m_1 < 0$. Here $W^+_\phi, W^-_\phi$ are contained in the neighborhoods from item (2) of Lemma \ref{disjointnbds} and have the same properties as those.

  \end{enumerate}
\end{lemma}

 \begin{proof}
 Proof of (1) follows from generalizing the Schottky-type argument in the three-of-four stretch Proposition \ref{34} using Corollary \ref{2k-1st}.

 To prove (2) note that the hallway flaring condition is proven in item (2) of Proposition \ref{34}.
  The cone-bounded strictly flare condition is Proposition \ref{fbfcbhfc}. We apply Lemma \ref{mjr} to get the conclusion and conclude the proof of (2).
  
 Conjugacy flaring condition for $Q$ (item (1) of Proposition \ref{34}) shows that every element of $Q$ is exponentially growing (relative to $\mathcal{K}$) \emph{i.e.} every conjugacy class not carried by $\mathcal{K}$ grows exponentially relative to $\mathcal{K}$ under iterates of every element of $Q$, hence will be weakly attracted to some attracting lamination not carried by $\mathcal{K}$. Further, since each $\phi_1, \ldots \phi_k $ leave each component of $\mathcal{K}$ invariant, all elements of $Q$ will also do do the same. This completes proof of second part of (2).
 
 \end{proof}

\begin{corollary}\cite[Theorem 5.2]{BFH-97}\label{bfhhyp}
  Suppose $\phi,\psi$ are irreducible  hyperbolic outer automorphisms which do not have a common power. Then there
  exists some $M>0$ such that for every $m, n \geq M$ the group $Q:=\langle \phi^m, \psi^n \rangle$ is a free group of
  rank $2$ and the extension group $\F\rtimes \widehat{Q}$ is word hyperbolic.
 \end{corollary}

 \begin{proof}
 Apply Theorem \ref{fbfrh} with $\mathcal{K}=\emptyset$ and $k=2$.
 \end{proof}

The following corollary has multiple appearances in the literature due to various authors. See \cite{GMj-22} for a more general form and \cite{Uya-17} for a similar result.

\begin{corollary}

Let $\phi, \psi\in \out$ be atoroidal and let $\mathcal{L}^+(\phi), \mathcal{L}^-(\phi) $ be collections of all attracting and repelling laminations, respectively  of $\phi$ and $\mathcal{L}^+(\psi), \mathcal{L}^-(\psi)$ be collections of all attracting and repelling laminations of $\psi$. Suppose that $\phi, \psi$  are independent relative to $\mathcal{K} = \emptyset$. 

Then there exists $M>0$ such that $Q=\langle \phi^m, \psi^n \rangle$ is a free group of rank two and the extension group $\F \rtimes \widehat{Q}$ is a hyperbolic group for any lift $\widehat{Q} \leq \aut$ of $Q$.
\end{corollary}

\begin{proof}
Apply Theorem \ref{fbfrh} with $\mathcal{K} = \emptyset$ and $k=2$.
\end{proof}

We end this section with a very special example, one that highlights the importance of the technique developed in this work.
\begin{example}
Let  $\F$ be a rank 6 free group generated by $a, b, c, d, e, f$ and let $F_1 = \langle a, b \rangle $ and let $F_2 = \langle c, d, e, f\rangle$. Choose $\phi,\psi\in \out$ in such a way that
\begin{enumerate}
\item $\phi([F_1]) = [F_1]$ and the restriction of $\phi$ to $F_1$ is polynomially growing.
\item $\phi$ is exponentially growing  and has a nongeometric dual lamination pair $\Lambda^\pm_\phi$ which which falls over to $F_1$, i.e. free factor support of $\Lambda^\pm_\phi$ is $[\F]$ and $\mathcal{A}_{na}(\Lambda^+_\phi) = [F_1]$.
\item $\psi([F_1]) = [F_1]$ and the restriction of $\psi$ to $[F_1]$ is fully irreducible and geometric.
\item $\psi$ preserves $[F_2]$ and its restriction is fully irreducible and nongeometric with dual lamination pair $\Lambda^\pm_\psi$ with its free factor support being $[F_2]$.
\end{enumerate}

\end{example}
Theorem \ref{fbfrh} tells us that for some sufficiently large $M$, the group $Q = \langle \phi^m, \psi^n\rangle$ is a free group of rank 2 if $m, n \geq M$. If one tries to find subgroups of $\F$ so that the restriction of $Q$ is polynomially growing and then cone-off to show relative hyperbolicity, it might not be feasible in this case.

But our conditions here imply that $[F_1] $ is an admissible subgroup system and the laminations $\Lambda^\pm_\phi, \Lambda^\pm_\psi$ do not have asymptotic generic leaves (since the free factor supports are distinct). Theorem \ref{fbfrh} now shows that $\F\rtimes \widehat{Q}$ hyperbolic relative to $F_1 \rtimes \widehat{Q}_1$, where $\widehat{Q}_1$ is a lift that preserves $F_1$ and $\widehat{Q}$ is any lift of $Q$.

\subsection{Construction of an admissible subgroup system: the non-attracting sink}\label{admissconst}

Now we will give the details of how to construct a malnormal subgroup system which satisfies  the hypotheses of Theorem  \ref{fbcrh} and of Theorem \ref{fbfrh}. The motivation behind the idea is the notion of the \emph{meet of free factor system} given in \cite[Section 2]{BFH-00} and in  Corollaries \ref{relhypext} and \ref{GL}. Note that an admissible subgroup system may not be unique, and our method shows one of the ways in which we can construct such a subgroup system.  

 The idea is that given $\phi\in\out$ which is  exponentially growing, the malnormal subgroup system given in Section \ref{SA} should intuitively be a finite intersection of closed sets so that the lines and conjugacy classes carried by the malnormal subgroup system will be carried by the nonattracting system of every element of $\mathcal{L}^+_\mathcal{K}(\phi)$.

\noindent \textbf{Constructing $\mathcal{K}$:}  Let $\psi\in\out$ be exponentially growing and $\Lambda_1, \Lambda_2$ be attracting laminations of $\psi$ and $\mathcal{K}_1, \mathcal{K}_2 $ be the nonattracting subgroup systems of $\Lambda_1, \Lambda_2$ respectively. We wish to describe a new malnormal subgroup system which is in a sense ``intersection''  of $\mathcal{K}_1, \mathcal{K}_2$, in the sense that a conjugacy class is carried by both $\mathcal{K}_1$ and $\mathcal{K}_2$ if and only if it is carried by the ``intersection''.

\begin{definition} The \emph{meet} $\mathcal{K}_1 \bigwedge \mathcal{K}_2$ of subgroup systems $\mathcal{K}_1 $ and $\mathcal{K}_2$ is given by,

\[ \mathcal{K}_1 \bigwedge \mathcal{K}_2  = \{ [A \cap B^w] : [A]\in \mathcal{K}_1, [B] \in \mathcal{K}_2, w\in \F, A \cap B^w \neq \{id\} \}\]
where $B^w$ denotes the group $w^{-1} B w$, when $B$ is a subgroup of $\F$ and $w\in \F$.
\end{definition}

 Let $[\langle \rho_1 \rangle]\in\mathcal{K}_1$ be such that $[c_1]$ is the conjugacy class representing the unique closed indivisible Nielsen path associated to $\Lambda_1$, when $\Lambda_1$ is geometric. Let $\Lambda_2$ be any other attracting lamination. Geometricity of lamination implies that $[\rho_1]$ is nontrivial, $[c_1]$ is carried by $\mathcal{F}_{supp}(\Lambda_1)$ (\cite[Proposition 2.18]{HM-20}) and $\Lambda_1$ is always a bottommost lamination (see part (3) of \cite[Proposition 2.15]{HM-20}). Bottommost property ensures that no generic leaf of $\Lambda_1$ is weakly attracted to $\Lambda_2$.
 So $\mathcal{F}_{supp}(\Lambda_1)$ is carried by $\mathcal{A}_{na}(\Lambda_2)$. When we take $\mathcal{K}_1 \bigwedge \mathcal{K}_2$, one of the components of this meet carries  $[\langle c_1 \rangle]$. In fact, once we prove malnormality of $\mathcal{K}_1 \bigwedge \mathcal{K}_2$, only one component of $\mathcal{K}_1 \bigwedge \mathcal{K}_2$ can carry $[c_1]$ and that component is not a free factor.

 To depict this situation, think of  a situation where some CT representing $\psi$ has exactly three strata and all of them are EG. Let $\Lambda_1, \Lambda_2, \Lambda_3$ be associated attracting laminations such that $\Lambda_1\subset \Lambda_2 \subset \Lambda_3$ with nonattracting subgroup systems $\mathcal{K}_1, \mathcal{K}_2, \mathcal{K}_3$. Suppose we choose to work with $\Lambda_1, \Lambda_3$ and compute $\mathcal{K}_1\bigwedge \mathcal{K}_3$. If $\Lambda_2$ was geometric and $\Lambda_1, \Lambda_3$ were non-geometric, then by our definition, $\mathcal{K}_1\bigwedge\mathcal{K}_3$ would be empty and we  would fail to trap the conjugacy class corresponding to the unique closed indivisible Nielsen path associated with $\Lambda_2$. Similarly if $\Lambda_1, \Lambda_2$ were both geometric, we would be in trouble if the one decided to take $\mathcal{K}_1\bigwedge \mathcal{K}_3$.  These are the situations which are prevented from happening because if $\Lambda_2$ is geometric, then it is necessarily bottommost.

We use the the notation $A \leq B$ to mean $A$ is a subgroup of $B$ (not necessarily proper), when $A, B$ are both subgroups of $\F$.

\begin{proposition}\label{meetofnas}
Let $\psi\in\out$ be exponentially growing and $\Lambda_1, \Lambda_2$ be attracting laminations of $\psi$. Let $\mathcal{K}_1, \mathcal{K}_2 $ be the nonattracting subgroup systems of $\Lambda_1, \Lambda_2$ respectively. If $\mathcal{K}_1 \bigwedge \mathcal{K}_2 \neq \{\text{id.}\}$, then
\begin{enumerate}
	\item $\mathcal{K}_1 \bigwedge \mathcal{K}_2$ is a malnormal subgroup system.
	\item A conjugacy class $\alpha$ is carried by $\mathcal{K}_1 \bigwedge \mathcal{K}_2$ if and only if it is carried by both $\mathcal{K}_1$ and $\mathcal{K}_2$.
	\item A conjugacy class $\alpha$ is not carried by $\mathcal{K}_1 \bigwedge \mathcal{K}_2$  if and only if $\alpha$ is weakly attracted to $\Lambda_1$ or $\Lambda_2$.
	\item $\mathcal{K}_1 \bigwedge \mathcal{K}_2$ is $\psi$-invariant.
	
\end{enumerate}
\end{proposition}

 \begin{proof} 
 Let $f:G\to G$ be a CT map for $\psi$ and  suppose $\Lambda_1$ has height $r$  and $\Lambda_2$ has height $s$. If $\Lambda_1$ is geometric, we let $\rho_1$ denote the unique closed indivisible Nielsen path of height $r$. Similarly define $\rho_2$. We use $[c_i]$ to denote the conjugacy class representing $\rho_i$, if $\rho_i$ is non-trivial.  Let $\mathcal{K}_i = \mathcal{F}_i \cup [F_{m_i}]$ (where $F_{m_i}$ is trivial if $\Lambda_i$ is nongeometric) for $i= 1, 2$ (see the last paragraph of the Section \ref{sec:6}).
 
 When $F_{m_i}$ is nontrivial, it is always carried by a single component $\mathcal{F}_j$, where $i\neq j$. To see why, assume for concreteness that $F_{m_1}$ is nontrivial. Recall that $[F_{m_1}]$ carries $[\langle c_1 \rangle]$. Firstly observe that, $\psi$ restricted to $F_{m_1}$ having polynomial growth implies that $F_{m_1}$ must be carried by $\mathcal{K}_2$. By Lemma \ref{NAS} - item (6), this means that a single component of $\mathcal{K}_2$ carries $[F_{m_1}]$. Supposing that component was $[F_{m_2}]$ (implying that $[\langle c_1 \rangle]$ is carried by $[F_{m_2}]$), we argue to a contradiction using heights. 
  If $r < s$, then 
 $[\langle c_1 \rangle]$ is carried by $\pi_1(G_{s-1}) \sqsubset \mathcal{F}_2$ and we break malnormality of $\mathcal{K}_2$. Next, suppose $r > s$, then $[\langle c_2 \rangle]$ is carried by $\pi_1(G_{r-1}) \sqsubset \mathcal{F}_1$. Hence $[F_{m_2}]$ (which also carries $[\langle c_1 \rangle]$) must be carried by $\mathcal{F}_1$ - this violates malnormality of $\mathcal{K}_1$. In conclusion, $[F_{m_1}]$ is carried by a single component of $\mathcal{F}_2$. 

It follows that $\mathcal{K}_1 \bigwedge \mathcal{K}_2 = \mathcal{F}_1\bigwedge \mathcal{F}_2 \cup [F_{m_1}] \cup [F_{m_2}]$. Using the construction of \cite[Section 2]{BFH-00}, we know that $\mathcal{F}_1\bigwedge \mathcal{F}_2 $ is a free factor system and  this shows that $\mathcal{K}_1 \bigwedge \mathcal{K}_2$ is a finite collection of conjugacy classes of finite rank subgroups of $\F$. Malnormality of $\mathcal{K}_1 \bigwedge \mathcal{K}_2$ now follows from the uniqueness of $\rho_1$ and $\rho_2$ and malnormality of $\mathcal{K}_1$ and $\mathcal{K}_2$.
  This proves  (1). \\
 A  conjugacy class $\alpha$ is weakly attracted to $\Lambda_i$ if and only if $\alpha$ is not carried by $\mathcal{K}_i$, so  (2) and (3) are equivalent statements. We shall prove (3). 
 
 Suppose $[g]$ is weakly attracted to $\Lambda_1$. If $[c_1]$ or $[c_2]$ are nontrivial, then $[c_1]\neq[g]\neq[c_2]$ as $[c_1], [c_2]$ are fixed by $\psi$. Since $[g]$ is weakly attracted to $\Lambda_1$, it cannot be carried by $\mathcal{K}_1$. So $[g]$ cannot be carried by any $[A]\in\mathcal{K}_1$. Hence $[g]$ cannot be carried by any $[A\cap B^w]$ where $[A]\in \mathcal{K}_1, [B]\in\mathcal{K}_2, w\in \F$. So $[g]$ cannot be carried by $\mathcal{K}_1 \bigwedge \mathcal{K}_2$. Similar argument works when $[g]$ is weakly attracted to $\Lambda_2$.

 For the other direction of item (3), suppose $[g]$ is not weakly attracted to both $\Lambda_1$ and $\Lambda_2$. This implies that $[g]$ is carried by both $\mathcal{K}_1$ and $\mathcal{K}_2$. We have four possibilities: 
 $[g]$ is carried by either $[F_{m_i}]$ or $\mathcal{F}_i$, $i=1,2$. Using $\mathcal{K}_1 \bigwedge \mathcal{K}_2 = \mathcal{F}_1\bigwedge \mathcal{F}_2 \cup [F_{m_1}] \cup [F_{m_2}]$, we are done.  
 This completes the proof of item (3).
 
 For the proof of item (4), consider the subgroup system $\psi(\mathcal{K}_1\bigwedge\mathcal{K}_2)$. We want to show that $\psi(\mathcal{K}_1\bigwedge\mathcal{K}_2) = \mathcal{K}_1\bigwedge\mathcal{K}_2$. Any conjugacy class $[g]$ is carried by $\psi(\mathcal{K}_1\bigwedge\mathcal{K}_2)$ if and only if $\psi^{-1}([g])$ is not weakly attracted to either $\psi(\Lambda_1) = \Lambda_1$  or $\psi(\Lambda_2)=\Lambda_2$, this happens if and only if $[g]$ is weakly attracted to neither $\Lambda_1$ nor $\Lambda_2$. By item (3) we then get that $[g]$ is carried by $\mathcal{K}_1\bigwedge\mathcal{K}_2$. Conversely, if $[g]$ is carried by $\mathcal{K}_1\bigwedge\mathcal{K}_2$, then $[g]$ is not weakly attracted to either $\Lambda_1$ or  $\Lambda_2$. This means $\psi^{-1}([g])$ is weakly attracted to neither $\Lambda_1$ nor $\Lambda_2$. By using item (3), we get that $\psi^{-1}[g]$ is carried by $\mathcal{K}_1\bigwedge\mathcal{K}_2$. Therefore, $[g]$ is carried by $\psi(\mathcal{K}_1\bigwedge\mathcal{K}_2)$. This shows that $\psi(\mathcal{K}_1\bigwedge\mathcal{K}_2) = \mathcal{K}_1\bigwedge\mathcal{K}_2$. Moreover, since $\psi$ is rotationless, we get $\psi([H]) = [H]$ for every $[H]\in \mathcal{K}_1\bigwedge\mathcal{K}_2$. This concludes proof of item (4).
 \end{proof}

In light of Proposition \ref{meetofnas}, we can now extend the notion of a meet of \emph{two} nonattracting subgroup systems to a collection of finitely many nonattracting subgroup systems recursively. The proof is a simple generalization of the proof of Proposition \ref{meetofnas}.

\begin{corollary}\label{sinkc}
    Let $\phi$ be exponentially growing. Let $\Lambda_1, \ldots, \Lambda_r$ be some collection of attracting laminations of $\phi$ with nonattracting subgroup systems $\mathcal{K}_1, \ldots \mathcal{K}_r$, respectively. Then the subgroup system $\mathcal{K}_1\bigwedge\mathcal{K}_2\bigwedge\ldots\bigwedge \mathcal{K}_r$ is malnormal and a conjugacy class $[g]$ is carried by this subgroup system if and only if $[g]$ is carried by the nonattracting subgroup system of each $\Lambda_q$, for $q \in  \{1\cdots, r\}$. 
    
    Moreover, $\mathcal{K} = \mathcal{K}_1\bigwedge\mathcal{K}_2\bigwedge\ldots\bigwedge \mathcal{K}_r$ is an admissible subgroup system for $\phi$. 
\end{corollary}

\begin{definition}(Sink of an automorphism)\label{sink} Given any exponentially growing outer automorphism $\phi$ of $\F$, consider the full list of attracting laminations $\{\Lambda_q\}_{q = 1}^r$ and let
\[\mathcal{K}^*_\phi = \mathcal{K}_1\bigwedge\mathcal{K}_2\bigwedge\ldots\bigwedge \mathcal{K}_r, \,\,\text{where} \,\,\mathcal{K}_q = \mathcal{A}_{na}(\Lambda_q) \,\,\text{for} \,\,q = 1, \ldots,  r.\] If $\phi$ is not exponentially growing define $\mathcal{K}^*_\phi = \{[\F]\}$.  We shall call the malnormal subgroup system $\mathcal{K}_\phi^*$  the \emph{nonattracting sink} of $\phi$.

\end{definition}
It directly follows that any nontrivial conjugacy class carried by the nonattracting sink of $\phi$ grows polynomially under iteration by $\phi$.
It is easy to check that for an exponentially growing $\phi$ the nonattracting sink satisfies all the conditions of our standing assumptions of section \ref{SA} and hence it is a choice for an admissible subgroup system for $\phi$. Moreover, $\mathcal{K}_\phi^*$ is always a proper subgroup system when $\phi$ is exponentially growing. We end this section with the following corollary, which is a restatement of a known theorem.

\begin{corollary}[Theorem 3.15, \cite{Gh-20}]
    Let $\phi\in\out$ and $\mathcal{K}^*_\phi$ be its nonattracting sink. Then $\F\rtimes \langle \phi \rangle$ is relatively hyperbolic $\Leftrightarrow$  $\mathcal{K}^*_\phi$ is a proper subgroup system. 
\end{corollary}

\section{From relatively hyperbolic free-by-free groups to admissible subgroup systems } \label {RHtoM}

In this section we will prove the second part of Theorem \ref{main1} and finish the proof. 

To prove the next theorem we have to switch back and forth between several equivalent definitions of relative hyperbolicity for groups. Osin gave a definition of relative hyperbolicity in terms of relative Dehn functions in \cite[Definition 2.35]{Osin-06}. This definition is equivalent to Bowditch's definition with ``fine graph actions" (and equivalently Farb's with BCP property) for finitely generated groups, as is our case. The  equivalence of Bowditch and  Farb's (with BCP) definitions can be found in the appendix of  Dahmani's thesis \cite[Proposition 2, Proposition 3]{Dah-03} (also see \cite[Theorem 5.1, page 1824]{Hr-10}) whereas equivalence of the definition of Farb (with BCP) with that of Osin is given by \cite{Osin-06}. Hruska's paper \cite{Hr-10} gives a good overview of the connections between these definitions and how they are equivalent in the countable groups case as well. 

In the theorem below, proof of  conclusion $(1a)$ uses Osin's definition. Moreover, the ``torsion-free" assumption on $Q$ is not a strong assumption since every subgroup of $\out$ has a torsion-free subgroup of finite index (see \cite[Theorem B]{HM-20}). Also, since $\F$ is itself torsion-free, it does not change the malnormal subgroup system we construct in $(1b)$.  

\begin{theorem}\label{necrelhyp}
    
Let $Q < \out$ be a finitely generated torsion-free subgroup  generated by a collection   $\phi_1, \cdots, \phi_k$ of exponentially growing  automorphisms such that $\phi_i, \phi_j$ do not have a common power whenever $i\neq j$. Consider the short exact sequence $1 \to \F \to E_Q \to  Q \to 1$. Suppose that $E_Q $ is finitely generated and  hyperbolic relative to some collection of subgroups $H_1, \cdots, H_p$. Then,
\begin{enumerate}
     \item [(1a)] $\{[H_1], [H_2], \ldots , [H_p]\}$ is a malnormal subgroup system in $E_Q$. 
    \item [(1b)] $\mathcal{K} = \{[H_1\cap \F], [H_2\cap \F], \ldots , [H_p \cap \F]\}$ is a (possibly empty) malnormal subgroup system in $\F$. 
    \item [(1c)] $\F$ is hyperbolic relative to the collection $\{K_s\}_{s=1}^p$, where $K_s = H_s\cap \F$. 
    \item [(1d)] For each $s$, $N_{E_Q}(K_s) = H_s$  where $N_{E_Q}(K_s)$ is the normalizer of $K_s$ in $E_Q$
\end{enumerate}
Moreover, if $E_Q$ preserves cusps, \emph{i.e.} for any $\theta\in E_Q$ there exists $w_s\in \F$ for each $s$, such that $\theta K_s \theta^{-1} = w_s K_s w_s^{-1}$. Then the following additional properties are true:
    \begin{enumerate}
    \item [(2a)] $Q$  preserves each $[K_s]$.
    \item [(2b)] There exists a q.i. section $\tau: Q \to E_Q$. Moreover, in the induced exact sequence $1 \to K_s \to H_s \to Q_s \to 1$, each $Q_s = Q$ and there exists a q.i. section from $Q \to H_s$. 
    \item [(2c)] $Q$ is a hyperbolic group. 
    \item  [(2d)] Non-attracting sink $\mathcal{K}^*_i$ of $\phi_i$ is carried by $\mathcal{K}$ for each $i\in \{1, \cdots, k\}$. 
    \item [(2e)] Each $\mathcal{L}^+_\mathcal{K}(\phi_i), \mathcal{L}^-_\mathcal{K}(\phi_i)$ is nonempty and $\mathcal{L}^\pm_\mathcal{K}(\phi_i) \cap \mathcal{L}_\mathcal{K}^\pm(\phi_j) = \emptyset$ if $i\neq j$. 
     \item [(2f)]  $\mathcal{K}$ is an admissible subgroup system for $Q$. 
    \end{enumerate}
\end{theorem}

\begin{proof}
 For the proof of $(1a)$, we use Osin's definition of relative hyperbolicity \cite[Definition 2.35]{Osin-06}. 
 Since $\F\rtimes \widehat{Q}$ is hyperbolic relative to the collection of subgroups $\{H_s\}_{s=1}^p$, by Osin's characterization (\cite[Definition 2.35]{Osin-06}) relative Dehn-function exists and it is linear. By \cite[Proposition 2.36]{Osin-06} and using the fact that $E_Q$ is torsion-free, we get (1a). 
 
   Let $[K_s] = [H_s\cap \F]$ and assume  $K_s\cap w^{-1} K_t w \neq \{e\} $ for some $w\in \F$ where $e$ is the identity element. This implies  $H_s \cap \F \cap w^{-1} H_t w \cap \F = H_s\cap w^{-1} H_t w \cap \F \neq \{e\}$.  Hence $H_s\cap w^{-1} H_t w \neq \{e\} $. By item (1a) we conclude that $w \in H_s\cap \F = K_s$ and $s = t$. This proves $(1b)$. 

   $(1c)$ follows immediately from $(1b)$. 
   
   To prove $(1d)$, observe that  if  $\xi\in E_Q $, then $\xi(K_s) = \xi K_s \xi^{-1} =  \xi (H_s\cap \F) \xi^{-1} = \xi H_s \xi^{-1} \cap \F$. So $\xi(K_s) = K_s \Leftrightarrow \xi\in H_s$. 
   
   $(2a)$ follows directly from the hypothesis that $E_Q$ preserves cusps. $(2b)$ follows from \cite[Theorem 2.10]{Pal-10} and \cite[Theorem 5.2]{MjS-16}. $(2c)$ follows from \cite[Proposition 5.7]{MjS-16}. 

   Our argument to prove $(2d)$ will be  similar to the proof of the converse direction in \cite[Proposition 5.17]{MjS-16}. Let $\mathcal{E}(E_Q , K_1,\ldots, K_p)$ be the partially  electrified  Cayley graph of $E_Q $ obtained by coning-off translates of Cayley graphs of $K_s'$ s in $E_Q $. Then $\mathcal{E}(E_Q , K_1,\ldots, K_p)$ is hyperbolic relative to translates of $Q_s (= Q)$ in $E_Q $ by the  Mj-Reeves combination theorem \cite[Theorem 4.7]{MjR-08} (see also \cite[Definition 1.49, Lemma 1.50]{MjS-16}). Since $Q$ is hyperbolic, we deduce that $\mathcal{E}(E_Q, K_1,\ldots, K_p)$ is  a hyperbolic metric space (\cite[Section 7]{Bow-97}).
   This gives us a hyperbolic metric graph bundle $\mathcal{E}(E_Q , K_1,\ldots, K_p)$ over $Q$. Hence by \cite[Proposition 5.8]{MjS-16} a flaring condition (\cite[Definition 1.12]{MjS-16}) is satisfied. Since flaring holds, $\mathcal{K}^*_i$ must be carried by $\mathcal{K}$ for all $i$ (as every conjugacy class carried by $\mathcal{K}^*_i$ grows at most polynomially under iteration by $\phi_i$ for each $i$). This completes proof of $(2d)$. 

    As seen in proof of $(2d)$, the metric graph bundle $\mathcal{E}(E_Q , K_1,\ldots, K_p)$ over $Q$ satisfies a flaring condition. Fix some $i$ arbitrarily. Any conjugacy class $[g]$, which is not carried by $\mathcal{K}$, must grow exponentially when iterated by any $\phi_i$, in the  electrified  metric. This implies that ${|| [g] ||}_{el}$ grows exponentially and hence $\mathcal L_\mathcal{K}([g])$ grows exponentially by Lemma \ref{comparison}. Hence, there must exist some exponentially growing strata and an associated lamination to which $[g]$ gets weakly attracted  under iteration by $\phi_i$, implying that $\mathcal{L}_\mathcal{K}^+(\phi_i)$ is nonempty (equivalently, some attracting lamination of $\phi_i$ is not carried by $\mathcal{K}$). For each attracting lamination in $\mathcal{L}_\mathcal{K}^+(\phi_i)$, its dual is also not carried by $\mathcal{K}$, hence $\mathcal{L}_\mathcal{K}^-(\phi_i)$ is also nonempty.

     Let $f: G\to G$ be a CT map for $\phi_i$ and assume  $\Lambda^+ \in \mathcal{L}^+_\mathcal{K}(\phi_i) \cap \mathcal{L}^+_\mathcal{K}(\phi_j) $  for some $j\neq i$. 
    Let $\ell$ be a generic leaf of $\Lambda^+$ which has height $r$. Attracting or repelling neighborhoods for $\Lambda^+$ are constructed by taking some sufficiently long segment of height $r$ in $\ell$. Pick some such finite leaf subsegment, say $\gamma_1$, of $\ell$. Since $G$ is a finite graph, by replacing $\gamma_1$ with a longer subsegment of $\ell$, if necessary, we may assume that $\gamma_1$ is a circuit in $G$. Choose $\gamma_2$ as a subsegment of $\ell$ which contains $\gamma_1$ as a subpath and has twice the length in the  electrified  metric. Arguing as before, we may choose $\gamma_2$ to be a circuit. Proceeding inductively, we build up a sequence of circuits $\{\gamma_k\}_k$ which converges to $\ell$ and every term itself is a subpath of $\ell$.  Using \cite[Corollary 1.4]{HM-19}  we have that the lamination dual to $\Lambda^+$, say $\Lambda^-$, is the  same for both $\phi_i, \phi_j$. Using the sequence of circuits it now follows that 3-outof-4 stretch lemma \ref{34} will fail no matter how large of an exponent we choose, since $\ell$, which a generic leaf of an attracting lamination,  cannot be weakly attracted to any repelling lamination of either $\phi_i$ or $\phi_j$. This concludes proof of $(2e)$. 

    To prove $(2f)$ we only need to show that every conjugacy class not carried by $\mathcal{K}$ is weakly attracted to some element of $\Lambda^+ \in \mathcal{L}^+_\mathcal{K}(\phi_i)$ for each $i$, the rest of the conditions follow from $(1b), (2a), (2e)$. Fix some $i$. Using the flaring condition for metric graph bundle, as in proof of $(2d)$, we can conclude that any conjugacy class $[g]$ not carried by $\mathcal{K}$ will have exponential growth, under iteration by $\phi_i$, in the  electrified  metric on $\widehat{\F}$ with cosets of $K_i$ coned-off.  Pick a $CT$ map for $\phi_i$, say $f: G\to G$. Using the setup from \ref{sec:3}, we get $L_\mathcal{K}([g])$ must also grow exponentially under iteration by $f_\#$. Thus, $[g]$ is weakly attracted to some attracting lamination in $\mathcal{L}^+_\mathcal{K}(\phi_i)$. This completes the proof.

\end{proof}

\begin{proof}[Proof of Theorem \ref{main1}]
To prove relative hyperbolicity we use Lemma \ref{fbfrh}. The other direction of the proof is by Theorem \ref{necrelhyp}.  
\end{proof}

\bibliographystyle{alpha}
\def\bibfont{\footnotesize}
\bibliography{biblo1}

\end{document}